\documentclass{amsart}

\usepackage{amssymb,amsmath}

\usepackage{appendix}

\usepackage[linktocpage]{hyperref}

\usepackage{geometry}
\numberwithin{equation}{section}

\newtheorem{theorem}{Theorem}[section]

\newtheorem{prop}{Proposition}[section]
\newtheorem{remark}{Remark}[section]

\newtheorem{question}{Question}[section]

\newcommand\pd{{\partial}}

\newcommand\pt{\partial_{\mathcal{T}}}

\newcommand{\R}{{\mathbb{R}}}

\newcommand\yinf{y\rightarrow+\infty}
\newcommand\zinf{z\rightarrow+\infty}

\newcommand\liy{\lim\limits_{y\rightarrow+\infty}}

\newcommand\ka{\kappa}
\newcommand\ep{\epsilon}
\newcommand\epz{\epsilon\rightarrow0}
\newcommand\sep{\sqrt{\epsilon}}

\newcommand\bu{\textbf{u}}
\newcommand\bU{\textbf{U}}
\newcommand\bb{\textbf{b}}
\newcommand\ta{\theta}
\newcommand\Ta{\Theta}

\newcommand\tta{\tilde\theta}

\renewcommand{\theequation}{\thesection.\arabic{equation}}

\begin{document}

\title[Thermal layer for inviscid compressible flows]
{Well-posedness of thermal layer equations for inviscid compressible flows}


\author{C.-J. Liu}
\address{Chengjie Liu
\newline\indent
School of Mathematical Sciences, Shanghai Jiao Tong University,
Shanghai, P. R. China
\newline\indent
and Department of Mathematics, City University of Hong Kong,
Hong Kong, P. R. China}
\email{cjliusjtu@gmail.com}

\author{Y.-G. Wang}
\address{Ya-Guang Wang
\newline\indent
School of Mathematical Sciences, MOE-LSC and SHL-MAC, Shanghai Jiao Tong University,
Shanghai, 200240, P. R. China}
\email{ygwang@sjtu.edu.cn}

\author{T. Yang}
\address{Tong Yang
\newline\indent
School of Mathematical Sciences, Shanghai Jiao Tong University,
Shanghai, P. R. China
\newline\indent
and Department of mathematics, City University of Hong Kong,
Hong Kong, P. R. China}
\email{matyang@cityu.edu.hk}



\date{}


\begin{abstract}
	A semi-explicit formula of solution to the boundary layer system for
	thermal layer derived from the
	compressible Navier-Stokes equations with the non-slip boundary condition
	when the viscosity coefficients vanish is given, in particular
	in three space dimension. In contrast to
	the inviscid Prandtl system studied by \cite{H-H} in two space dimension,
	the main difficulty comes
	from the coupling of the velocity field and the temperature field through a degenerate parabolic equation. The convergence of these boundary layer equations to the
	inviscid Prandtl system is justified when the initial temperature goes to a constant. Moreover, the time asymptotic
	stability of the linearized system around a shear flow is given, and in particular,
	 it shows that in three space dimension, the asymptotic stability depends on whether the direction of tangential velocity field of the shear flow is invariant in the normal direction respective to the boundary.
\end{abstract}

\maketitle

\tableofcontents


\section{Introduction}

There has been extensive study on the Prandtl equations since
Prandtl introduced in \cite{prandtl} to describe the behavior of flows near physical boundaries in viscous flows in 1904.
 The well-posedness theory and ill-posedness results obtained by Oleinik and her collaborators (\cite{O, Ole}), and Gerard-Varet, Dormy \cite{GD}, Grenier \cite{GR},  and Guo, Nguyen \cite{GN1} respectively,
  show
that the monotonicity of the tangential velocity in the normal direction
to the boundary plays an essential role in the well-posedness of the
Prandtl equations even locally in time. On the other hand, as observed by
van Dommnelen and Shen \cite{van} and studied mathematically by Hong and Hunter \cite{H-H}, the monotonicity
condition is not needed for the well-posedness of the inviscid Prandtl
equations at least locally in time.

This paper aims to study the
corresponding boundary layer problem derived from compressible Navier-Stokes
equations when the viscosity coefficients vanish or
are of higher order with respect to the heat conductivity coefficient,
i.e. the Prandtl number ${\rm Pr}$ is strictly smaller than one.
The results obtained in this paper
not only reveal the role of the temperature played in this boundary
layer system, but also reveal the phenomena in three space dimensions
that are different from those obtained by Hong and Hunter \cite{H-H} for two dimensional
inviscid Prandtl equations.

Precisely, we consider the following initial-boundary value problem in $\{(t,x',y):~t>0,x'\in\R^{d-1},y\in\R_+\}$ with $d=2,3$:
\begin{equation}\label{pr_invis}\begin{cases}
		\pd_t \bu_h+(\bu_h\cdot\nabla_h+u_d\pd_y)\bu_h=0,\\
		\pd_t \ta+(\bu_h\cdot\nabla_h+u_d\pd_y)\ta
		=\frac{\ka}{P}\ta \pd_y^2\ta+\frac{\ka P_t}{P}\ta,\\
		\nabla_h\cdot\bu_h+\pd_y u_d=\frac{\ka}{P} \pd_y^2\ta-\frac{(1-\ka)P_t}{P},\\
		(u_d,\ta)|_{y=0}=\big(0,\ta^0(t,x')\big),
		\quad\lim\limits_{\yinf}\ta(t,x,y)=\Ta(t,x'),\\
		(\bu_h,\ta)|_{t=0}=(\bu_{h0},\ta_0)(x',y),
	\end{cases}\end{equation}
	where $x'=(x_1,\cdots,x_{d-1}),~\nabla_h=(\pd_{x_1},\cdots,\pd_{x_{d-1}})^T;$ $\bu_h=(u_1,\cdots,u_{d-1})^T\in\R^{d-1}$ is unknown vector function, $u_d$ and $\ta$ are unknown scalar functions;    $P=P(t)$ and $\Ta(t,x')$ are positive known functions, and $\kappa>0$ is a constant.
	The above problem \eqref{pr_invis} discribes the behavior of boundary layer for inviscid compressible non-isentropic flow, as the heat conductivity tends to zero,
and the behavior of thermal layer for the compressible Navier-Stokes equations with nonslip boundary condition on velocity when the viscosity tends to zero faster than the heat conductivity.
The formal derivation of \eqref{pr_invis} will be given in the Appendix.

When the first equation of \eqref{pr_invis} has an additional diffusion term $\partial_y^2\bu_h$ on the right hand side, which describes the boundary layer behavior of the compressible Navier-Stokes equations, we
 have studied
  the well-posedness of this boundary layer problem in \cite{LWY4}, in  two space dimensions, under the usual monotonic condition on the tangential velocity with respect to the normal direction to the boundary, as for the classical incompressible Prandtl equations (\cite{O, Ole, AWXY, WXY, Mas-Wong, GR, GD}). In this paper,
motivated by the work of Hong and Hunter \cite{H-H}, we are going to study the problem \eqref{pr_invis} without the monotonicity of the tangential velocity.

	When the pressure of the outer flow is a function of time $t$ only, we will first
	give a semi-explicit formula for the solution to problem \eqref{pr_invis} in
	the next section.
In Subsection 2.2, we obtain that the velocity field of \eqref{pr_invis} converges to that of
the inviscid Prandtl system when the temperature
	tends to a constant state.
	And then in Section 3,  we will study the linearized system of \eqref{pr_invis}
	around a shear flow. In three space dimensions, it will be shown that
the
	solution to the linearized system is bounded for all positive $t$ when the tangential velocity direction of the background shear flow is independent of the normal direction to the boundary and one component of the tangential velocity is strictly monotonic with respect to the normal variable, and it grows like $\sqrt{t}$ when the
	two tangential components of the shear flow is not linearly dependent,
or they are linearly dependent and one component
has a non-degenerate critical point.

\section{Study of the nonlinear thermal layer problem}

\subsection{Local existence of classical solutions}

Before stating the local existence result, we first give some notations.
Denote by $I_{k}$ the $k\times k$ identity matrix for some integer $k$,  $det(A)$ the determinant of a matrix $A$, $\nabla_h u(x',\cdot)$ the gradient of a function $u$ with respect to the variables $x'\in \R^{d-1}$.
By using the intial data of the problem \eqref{pr_invis}, we introduce the vector function $\xi(t,x',z)\in\R^{d-1}$, defined by the following equation
\begin{equation}\label{2.1}
x'=\xi+t\bu_{h0}(\xi,z),
\end{equation}
and then, the functions $a(t,x',z)$  and $\bb(t,x',z)$ are defined as:
\begin{equation}\label{notation}
a(t,x',z):=\frac{P(t)}{P(0)}\ta_0\big(\xi(t,x',z),z\big)\cdot det(I_{d-1}+t\nabla_h\bu_{h0})\big(\xi(t,x',z),z\big),\quad \bb(t,x',z):=\bu_{h0}\big(\xi(t,x',z),z\big).
\end{equation}

Then, we have the following local existence of a classical solution to the problem \eqref{pr_invis}, in which no monotonicity condition is required on the initial data.
\begin{theorem}\label{thm-2-1}
Suppose that the data given in \eqref{pr_invis},  $\bu_{h0}\in C^2,\ta_0\in C^2,\ta^0\in C^1,P\in C^1$ and $\Ta\in C^1$ satisfy
the compatibility conditions of \eqref{pr_invis} up to order one, and
\begin{equation}\label{ass_ini}
t^*:=
\sup
\Big\{t: \inf\limits_{(x',y)\in\R^{d}_+}det\big(I_{d-1}+s\nabla_h\bu_{h0}(x',y)\big)>0,~\forall s\in[0,t]\Big\}>0.
\end{equation}
Also, there exists a positive constant $C_0$ such that for $t\in[0,t^*)$ and
$(x',y)\in\R_+^{d},$
\begin{equation}\label{ass_ta}
\begin{cases}
C_0^{-1}\leq \ta_0(x',y),~\ta^0(t,x'),~\Ta(t,x'),~P(t)\leq C_0,\\[3mm]
\|\bu_{h0}\|_{C^2}\le C_0,\quad \|\ta_0\|_{C^1}\leq C_0.	
\end{cases}
\end{equation}
Then, there exist a $t_0\in (0, t^*]$ and
a unique classical solution to \eqref{pr_invis} in $[0,t_0)\times\R^d_+$ given by
\begin{equation}\label{sol_invis}\begin{split}
&\bu_h(t,x',y)~=~\bu_{h0}\Big(\xi\big(t,x',\eta(t,x',y)\big),\eta(t,x',y)\Big),\\
& u_d(t,x',y)~=~\int_0^{\eta(t,x',y)}\pd_t(\frac{\tta}{a})(t,x',z)dz+\int_0^{\eta(t,x',y)}\Big[\bb\big(t,x',\eta(t,x',y)\big)\cdot\nabla_h(\frac{\tta}{a})(t,x',z)\Big]dz,\\
&\ta(t,x',y)~=~ \tta\big(t,x',\eta(t,x',y)\big).
\end{split}\end{equation}
Here, $a(t,x',z)$ and $\bb(t,x',z)$ are given by \eqref{notation}, 
$\tta(t,x',z)$ is a positive smooth solution to the following problem in
$[0,t_0)\times \R_+^d$:
\begin{equation}\label{pr_tta}\begin{cases}
\pd_t\tta+\bb\cdot\nabla_h\tta-\frac{\ka P_t}{P}\tta-\frac{\ka a}{P}\pd_z\big(\frac{a}{\tta}\pd_z\tta\big)=0,\\
\tta|_{z=0}=\ta^0(t,x'),
\quad \lim\limits_{\zinf}\tta=\Ta(t,x'),\\
\tta|_{t=0}=\ta_0(x',z),
\end{cases}\end{equation}
and the function $\eta(t,x',y)$ is defined
implicitly
 by the relation
\begin{equation}\label{tran_y}
y~=~\int_0^\eta\frac{\tta(t,x',z)}{a(t,x',z)}dz.
\end{equation}

\end{theorem}

\begin{proof}[\bf{Proof.}]

We shall use the method of characteristics, introduced in \cite{H-H} for the inviscid Prandtl equations, to get the solution
formula \eqref{sol_invis}
for the problem \eqref{pr_invis}.

(1) Suppose that $(\bu_h,u_d,\ta)(t,x',y)$ is a smooth solution to the problem \eqref{pr_invis}, we introduce characteristic coordinates:
\begin{equation}\label{tran}
t=\tau,~x'=x'(\tau,\xi,\eta),~y=y(\tau,\xi,\eta)
\end{equation}
being determined by
solving the problems,
\begin{equation}\label{pr_char}\begin{cases}
\frac{\partial}{\partial\tau}x'(\tau,\xi,\eta)~=~\bu_h
\big(\tau,x'(\tau,\xi,\eta),y(\tau,\xi,\eta)\big),\\
\frac{\partial }{\partial\tau}y(\tau,\xi,\eta)~=~u_d\big(\tau,x'(\tau,\xi,\eta),y(\tau,\xi,\eta)\big),\\
x'(0,\xi,\eta)~=\xi, ~y(0,\xi,\eta)~=~\eta,
\end{cases}\end{equation}
with $\xi=(\xi_1,\cdots,\xi_{d-1})^T\in\R^{d-1}$.

We denote by
\begin{equation}\label{new_fun}
(\bar \bu_h,\bar u_d,\bar\ta)(\tau,\xi,\eta)~:=~(\bu_h,u_d,\ta)\big(\tau,x'(\tau,\xi,\eta),y(\tau,\xi,\eta)\big),
\end{equation}
then, it is easy to deduce from \eqref{pr_invis} and the relation \eqref{pr_char} that $
(\bar \bu_h,\bar\ta)(\tau,\xi,\eta)$ satisfy the following problem:
\begin{equation}\label{eq_new}\begin{cases}
\pd_\tau\bar \bu_h~=~0,\\
\pd_\tau\bar\ta~=~\frac{\ka}{P(\tau)}\bar\ta ~\overline{\pd_y^2\ta}+\frac{\ka P_\tau(\tau) }{P(\tau)}\bar\ta,\\
(\bar \bu_h,\bar \ta)|_{\tau=0}~=~(\bu_{h0},\ta_0)(\xi,\eta)
\end{cases}\end{equation}
with the notation $\overline{\pd_y^2\ta}(\tau,\xi,\eta)=(\pd_y^2\ta)(\tau,x'(\tau,\xi,\eta),y(\tau,\xi,\eta))$.
We immediately obtain that from \eqref{eq_new},
\begin{equation}\label{formu_u}
\bar \bu_h(\tau,\xi,\eta)~\equiv~\bu_{h0}(\xi,\eta),
\end{equation}
which implies that by plugging \eqref{formu_u} into \eqref{pr_char},
\begin{equation}\label{formu_x}
x'~=~\xi+\tau \bu_{h0}(\xi,\eta).
\end{equation}
It is easy to see that the relation \eqref{formu_x} determines uniquely $\xi=\xi(\tau,x',\eta)$ when $0\leq\tau\le t^*$,
with $t^*>0$ being given in \eqref{ass_ini}.

(2)
Next, we are going to verify that the relation
$\eta=\eta(t,x',y)$ implicitly defined
 by \eqref{tran}-\eqref{pr_char} obeys the equation \eqref{tran_y}.
Denote by $J(\tau,\xi,\eta)$ the Jacobian of the transformation between $(x',y)$ and $(\xi,\eta)$:
\begin{equation}\label{j1}
J(\tau,\xi,\eta)~:=~\frac{\pd(x',y)}{\pd(\xi,\eta)}=\det(\nabla_\xi x')\cdot\pd_\eta y-\sum\limits_{i=1}^{d-1}\big[det(\nabla_i x')\cdot\pd_{\xi_i}y\big],\end{equation}
with the notation
\[
\nabla_i=( \cdots,\pd_{\xi_{i-1}},\pd_\eta,\pd_{\xi_{i+1}},\cdots)^T,\qquad 1\leq i\leq d-1.\]
By a direct computation and using \eqref{pr_char}, we get
\[
\pd_\tau J(\tau,\xi,\eta)=
J(\tau,\xi,\eta)\cdot(\nabla_h\cdot\bu_h+\pd_yu_d)
\big(\tau,x'(\tau,\xi',\eta),y(\tau,\xi',\eta)\big)
\]
which gives rises to
\begin{equation}\label{j}
\pd_\tau J(\tau,\xi,\eta)=
J(\tau,\xi,\eta)
\cdot\Big[\frac{\ka}{P(\tau)}\overline{\pd_y^2\ta}(\tau,\xi,\eta)-
(1-\ka)\frac{P_\tau(\tau)}{P(\tau)}\Big]
\end{equation}
by using
the third equation
 given in \eqref{pr_invis}. Note that $J(0,\xi,\eta)=1$,
thus combining \eqref{j} with the second equation given in \eqref{eq_new} we deduce
\begin{equation}\label{Jaco}
J(\tau,\xi,\eta)~=~\frac{P(0)}{P(\tau)\ta_0(\xi,\eta)}\bar\ta(\tau,\xi,\eta).
\end{equation}

Noting that
\[\det(\nabla_\xi x')(\tau,\xi,\eta)=det\Big(I_{d-1}+\tau\nabla_\xi \bu_{h0}(\xi,\eta)\Big)>0, \quad{\rm for}~\tau\leq t^*,\]
plugging \eqref{Jaco} into \eqref{j1} yields that
\begin{equation}\label{Jaco1}
\pd_\eta y-\sum\limits_{i=1}^{d-1}\Big[\frac{det(\nabla_i x')}{det(\nabla_\xi x')}\cdot\pd_{\xi_i}y\Big]=\frac{P(0)}{P(\tau)\ta_0(\xi,\eta)\cdot det(\nabla_\xi x')(\tau,\xi,\eta)}\bar\ta(\tau,\xi,\eta).
\end{equation}
By a direct calculation, it deduces that the characteristics of the equation \eqref{Jaco1} is $x'=constant$ or $\xi=\xi(\tau,x',\eta)$ given in \eqref{formu_x}.

Denote by
\begin{equation}\label{define_ta}
\tta(\tau,x',\eta)~:=~\bar\ta\big(\tau,\xi(\tau,x',\eta),\eta\big).
\end{equation}
From \eqref{Jaco1}, it follows
\begin{equation}\label{Jaco2}
\frac{\pd}{\pd \eta}y\big(\tau,\xi(\tau,x',\eta),\eta\big)~=~\frac{\tta(\tau,x',\eta)}{a(\tau,x',\eta)},
\end{equation}
where $a(\tau,x',\eta)$ is defined in \eqref{notation}.
Moreover, as $u_d|_{y=0}=0$, from \eqref{pr_char} we have $y=0$ when $\eta=0.$ Therefore, integrating \eqref{Jaco2} along characteristics, we obtain
\begin{equation}\label{formu_y}
y~=~y\big(\tau,\xi(\tau,x',\eta),\eta\big)~=~\int_0^\eta\frac{\tta(\tau,x',z)}{a(\tau,x',z)}dz.
\end{equation}
Consequently, when $0\leq\tau \le t^*$ and $\tta>0$, we have that $a>0$ from the definition \eqref{notation}, thus by using
\begin{equation*}
\pd_\eta y=\frac{\tta(\tau,x',\eta)}{a(\tau,x',\eta)}>0,
\end{equation*}
 the equation \eqref{formu_y} is invertible and  gives $\eta=\eta(\tau,x',y)$ with
 \begin{equation}\label{deri_eta}
\eta_y=\frac{a(\tau,x',\eta)}{\tta(\tau,x',\eta)}>0.
\end{equation}
Also, the domain $\{y>0\}$ is changed as $\{\eta>0\}$ with the boundary $\{y=0\}$, $y\rightarrow+\infty$ respectively, being changed as $\{\eta=0\},$ $\eta\rightarrow+\infty$ respectively.

(3) Now, we will derive the formula \eqref{sol_invis} and the problem \eqref{pr_tta} for $\tta(\tau,x',\eta)$. Note that the inverse function of
$x'=x'(\tau,\xi,\eta), ~y=y(\tau,\xi,\eta)$ given by \eqref{formu_x} and \eqref{formu_y},
 is
\[\Big(\xi\big(\tau,x',\eta(\tau,x',y)\big),\eta(\tau,x',y)
\Big).\]
Thus, combining \eqref{new_fun}, \eqref{formu_u} and \eqref{define_ta} yields that
\[
\bu_h(\tau,x',y)=\bu_{h0}\Big(\xi\big(\tau,x',\eta(\tau,x',y)\big),\eta(\tau,x',y)\Big),\quad \ta(\tau,x',y)=\tta\big(\tau,x',\eta(\tau,x',y)\big),
\]
which implies the formulas of $\bu_h(t,x',y)$ and $\ta(t,x',y)$ given in \eqref{sol_invis}.
Denote by
\[\tilde y(\tau,x',\eta)~:=~\int_0^\eta\frac{\tta(\tau,x',z)}{a(\tau,x',z)}dz,
\]
then from \eqref{formu_y} and \eqref{formu_x} we have $y(\tau,\xi,\eta)=\tilde y\big(\tau,\xi+\tau \bu_{h0}(\xi,\eta),\eta\big)$, 
which yields that
\begin{equation}\label{tv}
y_\tau(\tau,\xi,\eta)=\pd_\tau\tilde y\big(\tau,\xi+\tau \bu_{h0}(\xi,\eta),\eta\big)+\bu_{h0}(\xi,\eta)\cdot\nabla_h\tilde y\big(\tau,\xi+\tau \bu_{h0}(\xi,\eta),\eta\big).
\end{equation}
Combining \eqref{pr_char} with \eqref{tv}, we get that
\[\begin{split}
&u_d\big(\tau,x'(\tau,\xi,\eta),y(\tau,\xi,\eta)\big)=y_\tau(\tau,\xi,\eta)\\
&=\int_0^\eta\pd_\tau(\frac{\tta}{a})\big(\tau,\xi+\tau \bu_{h0}(\xi,\eta),z\big)dz+\int_0^\eta\bu_{h0}(\xi,\eta)\cdot\nabla_h(\frac{\tta}{a})\big(\tau,\xi+\tau \bu_{h0}(\xi,\eta),z\big)dz,
\end{split}\]
which implies the formula of $u_d(t,x,y)$ given in \eqref{sol_invis} by using that \eqref{formu_x} and \eqref{formu_y}.
Next, from \eqref{define_ta} and the relation \eqref{formu_x} we have
\[\bar\ta(\tau,\xi,\eta)~=~\tta\Big(\tau,\xi+\tau\bu_{h0}(\xi,\eta),\eta\Big),\]
which implies that,
\[\begin{split}
&\pd_\tau\bar\ta=\pd_\tau\tta+\bu_{h0}\cdot\nabla_{h}\tta,
\qquad\nabla_\xi\bar\ta=\big(I_{d-1}+\tau\nabla_\xi\bu_{h0}\big)\cdot\nabla_h\tta,\\
&\pd_\eta\bar\ta=\pd_\eta\tta+\tau\pd_\eta\bu_{h0}\cdot\nabla_h\tta.
\end{split}\]
Moreover, from \eqref{formu_x} it follows that
\[\big(I_{d-1}+\tau\nabla_\xi\bu_{h0}\big)\cdot\xi_\eta+\tau\pd_\eta\bu_{h0}=0,\]
thus we obtain that by virtue of \eqref{deri_eta},
\[
\begin{split}
\pd_y\bar\ta\big(\tau,\xi(\tau,x',\eta),\eta\big)
& = \Big[(\xi_\eta\cdot\nabla_\xi+\pd_\eta)\bar\ta\Big]
\big(\tau,\xi(\tau,x',\eta),\eta\big)\cdot\eta_y\\
&=(\pd_\eta\bar\ta-\tau\pd_\eta
\bu_{h0}\cdot \nabla_h\tta)\frac{a(\tau,x',\eta)}{\tta(\tau,x',\eta)}\\
&=\pd_\eta\tta(\tau,x',\eta)\cdot
\frac{a(\tau,x',\eta)}{\tta(\tau,x',\eta)}.
\end{split}\]
Therefore, the problem for $\bar\ta$ given in \eqref{eq_new} can be reduced as follows,
\[\begin{cases}
\pd_\tau\tta+\bu_{h0}\big(\xi(\tau,x',\eta),\eta\big)\cdot\nabla_h\tta=\frac{\ka a(\tau,x',\eta)}{P(\tau)}\pd_\eta\Big(\frac{a(\tau,x',\eta)}{\tta}\pd_\eta\tta\Big)+\frac{\ka P_\tau(\tau)}{P(\tau)}\tta,\quad {\rm in}~[0,t^*)\times\R_+^d,\\
\tta|_{\tau=0}=\ta_0(x',\eta).
\end{cases}\]
Furthermore, from the boundary conditions of $\ta$ given in \eqref{pr_invis}, we get
\[\tta|_{\eta=0}=\ta^0(\tau,x'),
\quad \lim\limits_{\eta\rightarrow+\infty}\tta(\tau,x',\eta)=\Ta(\tau,x'),\]
so we obtain the problem \eqref{pr_tta} for $\tta(\tau,x',\eta).$ Then, by the following Proposition \ref{lemma_ta}, we know that the problem \eqref{pr_tta} admits a unique classical solution in $[0,t_0)\times\R^d_+$ for some $0<t_0\leq t^*$. Finally, One can check directly that \eqref{sol_invis}-\eqref{tran_y} defines a smooth solution to the problem \eqref{pr_invis}.

\end{proof}

\begin{remark}
From \eqref{sol_invis}  and \eqref{tran_y} with the definition of the function $a(t,x',z) $ given in \eqref{notation}, one can see that there may be a loss of derivatives in the tangential variables $x'$ for the solution of \eqref{pr_invis},
 with respect to the regularity of the initial data.
\end{remark}


\begin{prop}\label{lemma_ta}
	Under the assumption of Theorem \ref{thm-2-1},  there is a time $0<t_0\leq t^*$ such that the problem \eqref{pr_tta} has a unique classical solution $\tta(t,x',z)$ with bounded derivatives in $[0,t_0)\times\R^d_+,$ satisfying that
	\begin{equation*}
		C_0^{-1-2\ka}\leq \tta(t,x',z) \leq C_0^{1+2\ka},
	\end{equation*}
for the constant $C_0$ given in \eqref{ass_ta}.
\end{prop}

\begin{proof}[\bf{Proof.}]
	Set
$$\hat{\ta}(t,x',z)=\frac{\tta(t,x',z)}{[P(t)]^\kappa}.$$
From the problem \eqref{pr_tta}, we know that $\hat{\ta}(t,x',z)$ satisfies the following initial-boundary value problem,
	\begin{equation}\label{pr_hta}\begin{cases}
	\pd_t\hat{\ta}+\bb\cdot\nabla_h\hat{\ta}-\frac{\ka a}{P^{1+\ka}}\pd_z\big(\frac{a}{\hat{\ta}}\pd_z\hat{\ta}\big)=0,\\
	\hat{\ta}|_{z=0}=\ta^0(t,x')/[P(t)]^\ka,
	\quad \lim\limits_{\zinf}\hat{\ta}=\Ta(t,x')/[P(t)]^\ka,\\
	\hat{\ta}|_{t=0}=\ta_0(x',z)/[P(0)]^\ka.
	\end{cases}\end{equation}

Introduce an auxiliary function $\varphi(s)$ defined for all $s\in\R$, satisfies that $\varphi$ is smooth, $\frac{1}{2}C_0^{-1-\ka}\leq \varphi\leq 2C_0^{1+\ka}$ with the positive constant $C_0$ given in \eqref{ass_ta}, and $\varphi(s)=\frac{1}{s}$ when $s\in\Big[C_0^{-1-\ka},C_0^{1+\ka}\Big]$. Then, corresponding to the problem
\eqref{pr_hta},
 we consider the following initial-boundary value problem,
	\begin{equation}\label{pr_aux}\begin{cases}
	\pd_t\ta+\bb\cdot\nabla_h\ta-\frac{\ka a}{P^{1+\ka}}\pd_z\big(a\varphi(\ta)\pd_z\ta\big)=0,\\
	\ta|_{z=0}=\ta^0(t,x')/[P(t)]^\ka,
	\quad \lim\limits_{\zinf}\ta=\Ta(t,x')/[P(t)]^\ka,\\
	\ta|_{t=0}=\ta_0(x',z)/[P(0)]^\ka.
	\end{cases}\end{equation}
Noting that the equation in \eqref{pr_aux} is degenerate parabolic with smooth coefficients, by employing the classical theory of degenerate parabolic equations (cf. \cite{solo}), we conclude there is a classical solution to the
problem
 \eqref{pr_aux} in $[0,t_0)\times\R^d_+$ for some $0<t_0\leq t^*.$

 Obviously, when the solution $\ta$ to \eqref{pr_aux} satisfies
\begin{equation}\label{bound_ta}
	C_0^{-1-\ka}\leq \ta(t,x',y)\leq C_0^{1+\ka},
\end{equation}
the problem \eqref{pr_aux} coincides with the one given in \eqref{pr_hta}. Thus, it suffices to verify \eqref{bound_ta} being true in the following.

To prove the lower bound of the solution $\ta(t,x',y)$ given in \eqref{bound_ta}, letting $r_+\triangleq\max\{r,0\}$, multiplying the equation of \eqref{pr_aux} by $\big(C_0^{-1-\ka}-\ta\big)_+$ and integrating over $\R_+^d$, it gives that by integration by parts,
\begin{equation}\label{low-bound}
\begin{split}
&-\frac{d}{2dt}\int_{\R_+^d}\big(C_0^{-1-\ka}-\ta\big)_+^2dx'dz+\int_{\R^d_+}\big[(\nabla\cdot \bb)\big(C_0^{-1-\ka}-\ta\big)_+^2\big]dx'dz\\
&=\int_{\R^d_+}\frac{\ka}{P^{1+\ka}} a\varphi(\ta)\pd_z\big(C_0^{-1-\ka}-\ta\big)_+\big[a\pd_z\big(C_0^{-1-\ka}-\ta\big)_++a_z\big(C_0^{-1-\ka}-\ta\big)_+\big]dx'dz.
\end{split}\end{equation}

It is obvious that the right hand side of the above equality satisfies
\[\begin{split}
&\int_{\R^d_+}\frac{\ka}{P^{1+\ka}} a\varphi(\ta)\pd_z\big(C_0^{-1-\ka}-\ta\big)_+\big[a\pd_z\big(C_0^{-1-\ka}-\ta\big)_++a_z\big(C_0^{-1-\ka}-\ta\big)_+\big]dx'dz\\
&\geq-\int_{\R^d_+}\frac{\ka\varphi(\ta)}{4P^{1+\ka}} \big[a_z\big(C_0^{-1-\ka}-\ta\big)_+\big]^2dx'dz,
\end{split}\]
thus, from \eqref{low-bound} we have
\[\begin{split}
\frac{d}{dt}\int_{\R_+^d}\big(C_0^{-1-\ka}-\ta\big)_+^2dx'dz&\leq\int_{\R^d_+}\big[2|\nabla\cdot \bb|+\frac{\ka\varphi(\ta)a_z^2}{2P^{1+\ka}}\big]\big(C_0^{-1-\ka}-\ta\big)_+^2dx'dz\\
&\leq C\int_{\R^d_+}\big(C_0^{-1-\ka}-\ta\big)_+^2dx'dz,
\end{split}\]
for a positive constant $C>0$. Applying the Gronwall inequality to the above expression and using that $(C_0^{-1-\ka}-\ta\big)_+|_{t=0}=0$, it yield that
\[
\int_{\R_+^d}\big(C_0^{-1-\ka}-\ta\big)_+^2(t,x',z)dx'dz=0,
\]
which implies that
\[
\ta(t,x',z)~\geq~C_0^{-1-\ka},
\]
and we obtain the lower bound of $\ta$ given in \eqref{bound_ta}.

The upper bound of the solution $\ta$ given in \eqref{bound_ta} can be obtained similarly, this gives rise to a classical solution to the problem \eqref{pr_hta}. Thus, the problem \eqref{pr_tta} admits a classical solution $\tta(t,x',z)=[P(t)]^\kappa\hat{\ta}(t,x',z)$, and the estimates \eqref{bd_ta} follows immediately. The uniqueness of the solution to \eqref{pr_tta} can be obtained by a standard comparison argument.
\end{proof}

\subsection{Convergence to the inviscid Prandtl equations}
In this subsection, we investigate the asymptotic behavior of the classical solution of \eqref{pr_invis} obtained in Theorem \ref{thm-2-1}, as $\ta$ tends to a positive constant.
For this, we consider a simple case of the problem \eqref{pr_invis} with a uniform outflow, i.e., the functions $P(t)$ and $\Ta(t,x')$ are constants,
and the general case can be studied similarly.
Consider the following problem
\begin{equation}\label{pr_shear}\begin{cases}
\pd_t \bu_h+(\bu_h\cdot\nabla_h+u_d\pd_y)\bu_h=0,\\
\pd_t \ta+(\bu_h\cdot\nabla_h+u_d\pd_y)\ta
=\ta \pd_y^2\ta,\\
\nabla_h\cdot\bu_h+\pd_y u_d= \pd_y^2\ta,\\
(u_d,\ta)|_{y=0}=\big(0,\ta^0(t,x')\big),
\quad\lim\limits_{\yinf}\ta(t,x,y)=1,\\
(\bu_h,\ta)|_{t=0}=(\bu_{h0},\ta_0)(x',y),
\end{cases}\end{equation}
and assume that
\begin{equation}\label{ass_con}
	\ta^0(t,x')~=~1+\ep\tilde\ta^0(t,x'),\quad \ta_0(x',y)~=~1+\ep\tilde\ta_0(x',y)
\end{equation}
with $\ep\ll1.$ Then, formally \eqref{pr_shear} tends to the following inviscid Prandtl system as $\ep\to 0$,
\begin{equation}\label{pr_con}\begin{cases}
\pd_t \bu_h+(\bu_h\cdot\nabla_h+u_d\pd_y)\bu_h=0,\\
\nabla_h\cdot\bu_h+\pd_y u_d= 0,\\
u_d|_{y=0}=0,\quad
\bu_h|_{t=0}=\bu_{h0}(x',y).
\end{cases}\end{equation}
For the above problem \eqref{pr_con}, through analogous arguments as given  in Theorem \ref{thm-2-1}, it's not difficult to obtain the following local existence of a
classical
solution.
\begin{prop}\label{prop_inP}
	Let $\bu_{h0}(x',y)$ be smooth and satisfy the compatibility conditions of the problem \eqref{pr_con} up to order one, and
$t^*>0$ be given as in \eqref{ass_ini}.
	Then,
the problem
\eqref{pr_con} has 
	a unique classical solution in $[0,t^*)$ given by
	\begin{equation}\label{sol_invis1}\begin{split}
	\bu_h(t,x',y)=&\bu_{h0}\Big(\xi_1\big(t,x',\eta_1(t,x',y)\big),
\eta_1(t,x',y)\Big),\\ u_d(t,x',y)=
&\int_0^{\eta_1(t,x',y)}\pd_t(\frac{1}{a_1})(t,x',z)dz+
\int_0^{\eta_1(t,x',y)}\Big[\bb_1\big(t,x',\eta_1(t,x',y)\big)
\cdot\nabla_h(\frac{1}{a_1})(t,x',z)\Big]dz,
	\end{split}\end{equation}
where, the vector function $\xi_1(t,x',z)\in\R^{d-1}$ is determined by the equation
	\begin{equation}\label{tran_x1}\begin{split}
	&x'~=~\xi_1+t\bu_{h0}(\xi_1,z);\\
	\end{split}\end{equation}
	the functions $a_1(t,x',z)$ and $\bb_1(t,x',z)$ are given as
	\begin{equation}\label{notation1}\begin{cases}
	a_1(t,x',z)~:=~ det(I_{d-1}+t\nabla_h\bu_{h0})\big(\xi_1(t,x',z),z\big),\\ 
	\bb_1(t,x',z)~:=~\bu_{h0}\big(\xi_1(t,x',z),z\big);
	\end{cases}\end{equation}
	and $\eta_1(t,x',y)$ is determined by the relation
	\begin{equation}\label{tran_y1}
	y~=~\int_0^{\eta_1}\frac{1}{a_1(t,x',z)}dz.
	\end{equation}
	
\end{prop}
Now, we show that the solution of \eqref{pr_shear} given in Theorem  \ref{thm-2-1} converges to $(\bu_h,u_d,1)$
when $\epsilon\to 0$,
where $(\bu_h,u_d)$ is the solution of \eqref{pr_con} given in Proposition \ref{prop_inP}.

\begin{theorem}\label{prop_con} Assume that the initial data of the problem \eqref{pr_shear} are smooth, and satisfy \eqref{ass_ini} and the compatibility conditions of \eqref{pr_shear} up to order one.
Moreover, they have the special form \eqref{ass_con}, with $\tilde{\ta}_0(x',y)$ satisfying
\begin{equation}\label{ass_tta}
(1+y)^k\tilde{\ta}_0(x',y)~\in~H^2_{x'}\big(\R^{d-1},H^1_y(\R_+)\big)
\end{equation}
for some constant $k>\frac{1}{2}$.
Let $(\bu_h,u_d,\ta)(t,x',y)$ ($0\le t< t_0\le t^*$) and $(\bu_{h1},u_{d1})(t,x',y)$ ($0\le t< t^*$) be the solutions of the problems \eqref{pr_shear}
 and \eqref{pr_con}  given in Theorem \ref{thm-2-1}  and Proposition \ref{prop_inP} respectively.
Then, for sufficiently small $\ep$ there is a constant $C>0$ independent of $\ep,$ such that for all $(t,x',y)\in[0,t_0)\times\R^d_+$,
\begin{equation}\label{est_con}	\big|(\bu_h,u_d,\ta)(t,x',y)-(\bu_{h1},u_{d1},1)(t,x',y)\big|~\leq~C\ep.
\end{equation}
\end{theorem}

\begin{proof}[\bf{Proof.}] (1)
	Setting the solution $\ta$ of the problem \eqref{pr_shear} having the form
	\begin{equation}\label{ta}
		\ta(t,x',y)~=~1+\ep\tilde{\ta}\big(t,x',\eta(t,x',y)\big),
	\end{equation}
then, from \eqref{pr_tta} we know that $\tilde{\ta}(t,x',z)$ satisfies the following problem  in $\{0\leq t<t_0\le t^*,x'\in\R^{d-1},z>0\}$,
\begin{equation}\label{tta}\begin{cases}
\pd_t\tta+\bb\cdot\nabla_h\tta- a\pd_z\big(\frac{a}{1+\ep\tta}\pd_z\tta\big)=0,\\
\tta|_{z=0}=\tilde\ta^0(t,x'),
\quad \lim\limits_{\zinf}\tta=0,\\
\tta|_{t=0}=\tilde\ta_0(x',z).
\end{cases}\end{equation}
Through similar arguments as given in the proof of Proposition \ref{lemma_ta}, we have the local existence of a classical solution to \eqref{tta}. Moreover, under the assumption \eqref{ass_tta}, by the standard energy method it's not difficult to obtain that there is a constant $C_1>0$ independent of $\ep,$ such that
\begin{equation}\label{bound_tta}
	\|(1+z)^k\tta\|_{L^\infty(0, t_0;H^2_{x'}(\R^{d-1}, H^1_z(\R_+)))}~\leq ~C_1,
\end{equation}
which implies the assertion \eqref{est_con} for the $\ta$ component by using the Sobolev embedding inequality.   	
	
(2) Comparing Theorem \ref{thm-2-1} with Proposition \ref{prop_inP}, we know that the auxiliary function $\xi(t,x',z)$ given by \eqref{2.1} coincides with $\xi_1(t,x',z)$ given by \eqref{tran_x1}, which implies that by combining \eqref{notation} with \eqref{notation1},
	\begin{equation}\label{b}
	a(t,x',z)~=~a_1(t,x',z)\big[1+\ep\tilde{\ta}_0\big(\xi(t,x',z),z\big)\big],\quad	\bb(t,x',z)~=~\bb_1(t,x',z).
	\end{equation}
Also, from \eqref{tran_y} and \eqref{tran_y1} we have
	\begin{equation*}
		\int_{0}^{\eta(t,x',z)}\frac{1+\ep\tilde{\ta}(t,x',z)}{a(t,x',z)}dz~=~\int_{0}^{\eta_1(t,x',z)}\frac{1}{a_1(t,x',z)}dz,
	\end{equation*}
which implies that by \eqref{b},
	\begin{equation}\label{est_eta}
		\ep\int_{0}^{\eta(t,x',z)}\frac{\tilde{\ta}(t,x',z)-\tilde{\ta}_0\big(\xi(t,x',z),z\big)}{a(t,x',z)}dz~=~\int_{\eta(t,x',z)}^{\eta_1(t,x',z)}\frac{1}{a_1(t,x',z)}dz.
	\end{equation}
Note that both $a$ and $a_1$ are bounded and have positive lower bounds, that is, there is a constant $C_2$ independent of $\ep$ such that
\[
C_2^{-1}~\leq~a(t,x',z),~a_1(t,x',z)~\leq ~C_2,\quad (t,x',z)\in [0,t_0)\times\R^d_+,\]
then the right-hand side of \eqref{est_eta} gives that
	\begin{equation}\label{est_rh} \Big|\int_{\eta(t,x',z)}^{\eta_1(t,x',z)}\frac{1}{a_1(t,x',z)}dz\Big|
~\geq~\frac{|\eta(t,x',z)-\eta_1(t,x',z)|}{C_2}.
	\end{equation}
On the other hand, we have that for the left-hand side term of \eqref{est_eta},
	\begin{equation*}\begin{split}
	&\Big|\int_{0}^{\eta(t,x',z)}\frac{\tilde{\ta}(t,x',z)-\tilde{\ta}_0\big(\xi(t,x',z),z\big)}{a(t,x',z)}dz\Big|\\
	&\leq C_2\big(\|\tilde{\ta}(t,x',z)\|_{L^1_z}+\|\tilde{\ta}_0\big(\xi(t,x',z),z\big)\|_{L^1_z}\big)\\
	&\leq C_3\Big(\|(1+z)^k \tilde{\ta}(t,x',z)\|_{L^2_z}+\|(1+z)^k\tilde{\ta}_0\big(\xi(t,x',z),z\big)\|_{L^2_z}\Big).
	\end{split}\end{equation*}
Note that for $t\in[0,t_0)$ and $x'\in\R^{d-1}$,
\[|\tilde{\ta}(t,x',z)|\leq\|\tilde{\ta}(t,x',z)\|_{H^2_{x'}},\]
and
\[|\tilde{\ta}_0\big(\xi(t,x',z),z\big)|\leq\|\tilde{\ta}_0\big(\xi(t,x',z),z\big)\|_{H^2_{x'}}\leq C_4\|\tilde{\ta}_0(x',z)\|_{H^2_{x'}},\]
where we use that $\xi(t,x',z)$ has bounded derivatives up to order two. From the above three inequalities we obtain that
\begin{equation}\label{est_lh}
	\Big|\int_{0}^{\eta(t,x',z)}\frac{\tilde{\ta}(t,x',z)-\tilde{\ta}_0\big(\xi(t,x',z),z\big)}{a(t,x',z)}dz\Big|\leq C_5\|(1+z)^k\tilde{\ta}\|_{L^\infty_t(H^2_{x'}L^2_z)}
\end{equation}
for some constant $C_5>0$ independent of $\ep.$
Plugging \eqref{est_rh} and \eqref{est_lh} into \eqref{est_eta}, it follows that
\begin{equation}\label{con_eta} |\eta(t,x',z)-\eta_1(t,x',z)|~\leq~C_2C_5\ep~ \|(1+z)^k\tilde{\ta}\|
_{L^\infty(0, t_0; H^2_{x'}(\R^{d-1}, L^2_z(\R_+)))}
\end{equation}
for all $(t,x',z)\in[0,t_0)\times\R^d_+$.

(4) Now we prove the estimate \eqref{est_con} for the components  $\bu_h$ and $u_d$. Since $\xi(t,x',z)=\xi_1(t,x',z)$, it follows that from the formulas of $\bu_h$ and $\bu_{h1}$ given by \eqref{sol_invis} and \eqref{sol_invis1} respectively,
\begin{equation}\label{con_u}\begin{split}
&|\bu_h(t,x',y)-\bu_{h1}(t,x',y)|\\
&\leq\|\xi_z(t,x',z)\cdot\nabla_h\bu_{h0}\big(\xi(t,x',z),z\big)
+\pd_y\bu_{h0}\big(\xi(t,x',z),z\big)\|_
{L^\infty}\cdot|\eta-\eta_1|(t,x',y),
\end{split}
\end{equation}
which implies that by using \eqref{con_eta},
\begin{equation}
	|\bu_h(t,x',y)-\bu_{h1}(t,x',y)|~\leq~C_6~\ep,
\end{equation}
for some constant $C_6>0$ independent of $\ep$. 
 Similarly, we can show \eqref{est_con} for the component $u_d$.
\end{proof}


\section{Linearized problems of thermal layer equations at a shear flow}

In this section, we study the well-posedness and long-time asymptotic behavior of the linearized problem of \eqref{pr_shear} at a shear flow.
It is easy to know that under proper initial and boundary data, \eqref{pr_shear} has a shear flow solution:
\begin{equation}\label{shear}
	(\bu_h,u_d,\ta)(t,x',y)~=~\Big(\bU_h(y),0,1\Big)
\end{equation}
with $\bU_h(y)=\big(U_1,\cdots,U_{d-1}\big)^T(y)$.
Then,
the linearized problem of \eqref{pr_shear} at the shear flow \eqref{shear} is given as
\begin{equation}\label{pr_linear}
  \begin{cases}
     &\pd_t\bu_h+\bU_h(y)\cdot\nabla_h\bu_h+\bU_h'(y)u_d=0,\\
     &\pd_t\ta+\bU_h(y)\cdot\nabla_h\ta=\pd_y^2\ta,\\
 &\nabla_h\cdot\bu_h+\pd_yu_d=\pd_y^2\ta, \\
     &(u_d,\ta)|_{y=0}=0,\\
     & (\bu_h,\ta)|_{t=0}=(\bu_{h0},\ta_0)(x',y).
  \end{cases}
\end{equation}

We observe that the problem \eqref{pr_linear} shall be solved by the following two steps. Firstly, we determine $\ta(t,x',y)$ by solving the linear initial-boundary value problem
in
$\{t>0, x'\in \R^{d-1}, y>0\}$:
\begin{equation}\label{pr_ta}
  \begin{cases}
    &\pd_t\ta+\bU_h(y)\cdot\nabla_h\ta=\pd_y^2\ta,\\
     &\ta|_{y=0}=0,\quad\ta|_{t=0}=\ta_0(x',y).
  \end{cases}
\end{equation}
Then, $(\bu_h(t,x',y), u_d(t,x',y))$ are obtained by studying the following problem for the linearized inviscid Prandtl type equations:
\begin{equation}\label{pr_u}
  \begin{cases}
    &\pd_t\bu_h+\bU_h(y)\cdot\nabla_h\bu_h+\bU_h'(y)u_d=0,\\
     &\nabla_h\cdot\bu_h+\pd_yu_d=\pd_y^2\ta, \\
     &u_d|_{y=0}=0,\quad \bu_h|_{t=0}=\bu_{h0}(x',y).
  \end{cases}
\end{equation}
Moreover, it's easy to know that the problem \eqref{pr_ta} with smooth and compatible initial data has a global classical solution and the solution is unique.

\subsection{
Explicit representations of solutions}

Based on the above discussion, we have the following result for the problem \eqref{pr_linear}.

\begin{prop}\label{prop_linear}
   Assume that $\bU_h(y),\bu_{h0}(x',y)$ and $\ta_0(x',y)$ are smooth,   and satisfy the compatibility conditions of \eqref{pr_linear} up to order one. Then, there exists a classical solution $(\bu_h,u_d,\ta)(t,x',y)$ to the problem \eqref{pr_linear}, where $\ta(t,x',y)$ is solved from the problem \eqref{pr_ta}, and $\bu_h(t,x',y), ~u_d(t,x',y)$ are given explicitly as
   \begin{equation}\label{formu-u}
     \begin{split}
        \bu_h(t,x',y)= & \bu_{h0}\big(x'-t\bU_h(y),y\big)+t\bU_h'(y)\int_{0}^{y}(\nabla_h\cdot \bu_{h0})\big(x'-t\bU_h(z),z\big)dz \\
          & +\bU_h'(y)\int_{0}^{y}\ta_0\big(x'-t\bU_h(z),z\big)dz-\bU_h'(y)\int_{0}^{y}\ta\big(t,x',z\big)dz,\\
       u_d(t,x',y) =  &\ta_y(t,x',y)-\ta_y(t,x',0) -\int_{0}^{y}\Big\{(\nabla_h\cdot\bu_{h0})\big(x'-t\bU_h(z),z\big)dz\\
       &-t\int_{0}^{y}\big[\bU_h(y)-\bU_h(z)\big]\cdot\nabla_h(\nabla_h\cdot\bu_{h0})\big(x'-t\bU_h(z),z\big)\Big\}dz\\
          & -\int_{0}^{y}\Big\{\big[\bU_h(y)-\bU_h(z)\big]\cdot\nabla_h\ta_0\big(x'-t\bU_h(z),z\big)\Big\}dz\\
          &+\int_{0}^{y}\Big\{\big[\bU_h(y)-\bU_h(z)\big]\cdot\nabla_h\ta\big(t,x',z\big)\Big\}dz.
     \end{split}
   \end{equation}
\end{prop}
\begin{proof}[\bf{Proof.}]
According to the arguments given before this proposition, we only need to derive the representations \eqref{formu-u} of $(\bu_h,u_d)(t,x',y)$. Denote by $\tilde f(s,\xi,y)$  the Fourier-Laplace transform of a function $f(t,x',y)$ for $t>0$ and $x'\in \R^{d-1}$,
\begin{equation}\label{Four_Lap}
	\tilde f(s,\xi,y)~:=~\int_{0}^{+\infty}\int_{\R^{d-1}}f(t,x',y)e^{-st-i\xi\cdot x'}dx'dt
\end{equation}
with ${\rm Re}~s>0$ and $\xi\in\R^{d-1}$.

Applying the Fourier-Laplace transform to the problem \eqref{pr_linear} yields that
\begin{equation}\label{pr_FL}
\begin{cases}
\big[s+i\xi\cdot\bU_h(y)\big]~\widetilde{\bu_h}-\widehat{\bu_{h0}}+\widetilde{u_d}~\bU_h'(y)=0,\\
\big[s+i\xi\cdot\bU_h(y)\big]~\tilde{\ta}-\widehat{\ta_{0}}-\pd_y^2\tilde{\ta}=0,\\
i\xi\cdot\widetilde{\bu_h}+\pd_y\widetilde{u_d}-\pd_y^2\tilde{\ta}=0,\\
\widetilde{u_d}(s,\xi,0)=\tilde{\ta}(s,\xi,0)=0,
\end{cases}
\end{equation}
where $\widehat{\bu_{h0}}=\widehat{\bu_{h0}}(\xi,y)$ and $\widehat{\ta_0}=\widehat{\ta_0}(\xi,y)$ are the Fourier transform of the initial data $\bu_{h0}(x',y)$ and $\ta_0(x',y)$ with respect to $x',$ respectively.

From the first and third equations of \eqref{pr_FL} we have
\[
\big[s+i\xi\cdot\bU_h(y)\big]~\pd_y\widetilde{u_d}-\big[i\xi\cdot\bU'_h(y)\big]~\widetilde{u_{d}}=
\big[s+i\xi\cdot\bU_h(y)\big]\pd_y^2\tilde{\ta}-i\xi\cdot\widehat{\bu_{h0}}.
\]
Solving this equation with the boundary condition $\widetilde{u_d}(s,\xi,0)=0,$ it follows that
\begin{equation}\label{FL_v}\begin{split}
\widetilde{u_d}(s,\xi,y)~&=~\int_0^y\frac{s+i\xi\cdot
\bU_h(y)}{s+i\xi\cdot\bU_h(z)}\pd_y^2\tilde{\ta}(s,\xi,z)dz
-\int_{0}^{y}\frac{s+i\xi\cdot\bU_h(y)}{\big[s+i\xi\cdot
\bU_h(z)\big]^2}~\big[i\xi\cdot\widehat{\bu_{h0}}(\xi,z)
\big]dz,
\end{split}\end{equation}
which implies
\begin{equation}\label{FL_v1}\begin{split}
\widetilde{u_d}(s,\xi,y)
~&=~\pd_y\tilde{\ta}(s,\xi,y)-\pd_y\tilde{\ta}(s,\xi,0)+\int_{0}^{y}\frac{i\xi\cdot\big[\bU_h(y)-\bU_h(z)\big]}{s+i\xi\cdot\bU_h(z)}\pd_y^2\tilde{\ta}(s,\xi,z)dz\\
&\quad -\int_{0}^{y}\Big\{\frac{1}{s+i\xi\cdot\bU_h(z)}+\frac{i\xi\cdot\big[\bU_h(y)-\bU_h(z)\big]}{\big[s+i\xi\cdot\bU_h(z)\big]^2}\Big\}~\big[i\xi\cdot\widehat{\bu_{h0}}(\xi,z)\big]dz\\
~&=~\pd_y\tilde{\ta}(s,\xi,y)-\pd_y\tilde{\ta}(s,\xi,0)+\int_{0}^{y}\big[i\xi\cdot\big(\bU_h(y)-\bU_h(z)\big)\big]\cdot\big[\tilde{\ta}(s,\xi,z)-\frac{\widehat{\ta_0}(\xi,z)}{s+i\xi\cdot\bU_h(z)}\big]dz\\
&\quad -\int_{0}^{y}\Big\{\frac{1}{s+i\xi\cdot\bU_h(z)}+\frac{i\xi\cdot\big(\bU_h(y)-\bU_h(z)\big)}{\big[s+i\xi\cdot\bU_h(z)\big]^2}\Big\}~\big[i\xi\cdot\widehat{\bu_{h0}}(\xi,z)\big]dz,
\end{split}
\end{equation}
by using the second equation given in \eqref{pr_FL}. Then, inverting the Fourier-Laplace transform in \eqref{FL_v1} we obtain the expression of $u_d(t,x',y)$ given in \eqref{formu-u}.

Plugging the relation \eqref{FL_v} into the first equation of \eqref{pr_FL} and using the second equation of \eqref{pr_FL}, we get
\begin{equation}\label{FL_u}
	\widetilde{\bu_h}(s,\xi,y)~=~\frac{\widehat{\bu_{h0}}(\xi,y)}{s+i\xi\cdot\bU_h(y)}+\Big\{\int_{0}^{y}\frac{i\xi\cdot\widehat{\bu_{h0}}(\xi,z)}{\big[s+i\xi\cdot\bU_h(z)\big]^2}dz+\int_{0}^{y}\frac{\widehat{\ta_0}(\xi,z)}{s+i\xi\cdot\bU_h(z)}dz-\int_{0}^{y}\tilde{\ta}(s,\xi,z)dz\Big\}~\bU_h'(y).
\end{equation}
Then, by inverting the Fourier-Laplace transform in this equality we deduce the expression of $\bu_h(t,x',y)$ given in \eqref{formu-u} immediately.

\end{proof}

\begin{remark} (1)
	One can also obtain the expression \eqref{formu-u} by solving the problem \eqref{pr_u} through the method of characteristics as introduced in \cite{H-H}.

	(2)
 From the expression of $\bu_h(t,x',y)$ given in \eqref{formu-u}, we know that
when the initial data $\bu_{h0}(x',y)$ decays faster than the background shear flow $\bU_h(y)$ as $y\to +\infty$, the decay rate of the solution
$\bu_h(t,x',y)$ of the  linearized problem \eqref{pr_linear} is mainly dominated by that of  $\bU_h'(y)$ when $y\rightarrow+\infty$.

(2)
The representation \eqref{formu-u} given in Proposition \ref{prop_linear} shows that in general, there is a loss of derivatives with respect to the tangential variables $x'$ for the solution $(\bu_h, u_d)(t,x',y)$ of the problem \eqref{pr_linear}.
\end{remark}

From the expression \eqref{formu-u}, we divide the solution $(\bu_h,u_d)(t,x',y)$ into two parts:
	\begin{equation}\label{decom}
		(\bu_h, u_d)(t,x',y)~:=~(\tilde\bu_h, \tilde u_d)(t,x',y)+(\bar\bu_h, \bar u_d)(t,x',y),
	\end{equation}
	where
\begin{equation}\label{formu-u1}\begin{cases}
\tilde\bu_h(t,x',y)= & \bu_{h0}\big(x'-t\bU_h(y),y\big)+t\bU_h'(y)\int_{0}^{y}(\nabla_h\cdot \bu_{h0})\big(x'-t\bU_h(z),z\big)dz ,\\
\tilde u_{d}(t,x',y) =  & -\int_{0}^{y}\Big\{(\nabla_h\cdot\bu_{h0})\big(x'-t\bU_h(z),z\big)dz\\
&-t\int_{0}^{y}\big[\bU_h(y)-\bU_h(z)\big]\cdot\nabla_h(\nabla_h\cdot\bu_{h0})\big(x'-t\bU_h(z),z\big)\Big\}dz,\\
\end{cases}\end{equation}
and
\begin{equation}\label{formu-u2}
\begin{cases}
\bar\bu_h(t,x',y)=
& \bU_h'(y)\int_{0}^{y}\ta_0\big(x'-t\bU_h(z),z\big)dz-\bU_h'(y)\int_{0}^{y}\ta\big(t,x',z\big)dz,\\
\bar u_d(t,x',y) =  &\ta_y(t,x',y)-\ta_y(t,x',0) 
 -\int_{0}^{y}\Big\{\big[\bU_h(y)-\bU_h(z)\big]\cdot\nabla_h\ta_0\big(x'-t\bU_h(z),z\big)\Big\}dz\\
&+\int_{0}^{y}\Big\{\big[\bU_h(y)-\bU_h(z)\big]\cdot\nabla_h\ta\big(t,x',z\big)\Big\}dz.
\end{cases}\end{equation}
Then, it is easy to know that  $(\tilde\bu_h,\tilde u_d)(t,x',y)$ and $(\bar\bu_h, \bar u_d)(t,x',y)$ satisfy the following intial-boundary value problems, respectively,
\begin{equation}\label{invis_prandtl}
	\begin{cases}
	&\pd_t\tilde\bu_h+\bU_h(y)\cdot\nabla_h\tilde\bu_h+\bU_h'(y)\tilde u_d=0,\\
	&\nabla_h\cdot\tilde\bu_h+\pd_y\tilde u_d=0, \\
	&\tilde u_d|_{y=0}=0,\quad \tilde\bu_h|_{t=0}=\bu_{h0}(x',y),
	\end{cases}
\end{equation}
and 
\begin{equation}\label{pr_bar}
\begin{cases}
&\pd_t\bar\bu_h+\bU_h(y)\cdot\nabla_h\bar\bu_h+\bU_h'(y)\bar u_d=0,\\
&\nabla_h\cdot\bar\bu_h+\pd_y\bar u_d=\pd_y^2\ta, \\
&\bar u_d|_{y=0}=0,\quad \bar\bu_h|_{t=0}=0.
\end{cases}
\end{equation}
Moreover,  we note that \eqref{invis_prandtl} is the linearization of the inviscid Prandtl equations at the shear flow $\big(\bU_h(y),0\big)$.

Denote by
\begin{equation}\label{def_h}
  \|\bu_h\|(t,y)~:=~\Big(\int_{\R^{d-1}}|\bu_h(t,x',y)|^2dx'\Big)^{\frac{1}{2}},
\end{equation}
and the following anisotropic space:
\[
L^{p,q}~:=~\{f=f(x',y)~\mbox{measurable}:~\|f\|_{L^{p,q}}
:=\|\|f\|_{L^p(dx')}\|_{L^q(dy)}<\infty\}
\]
for $1\leq p,q\leq\infty$, and
$$
H^{m,k}~:=~\{f=f(x',y)~\mbox{measurable}:~\|f\|_{H^{m,k}}:=\Big(\sum\limits_{|\alpha|\leq m,0\leq i\leq k}\|\pd_{x'}^\alpha\pd_y^if\|^2_{L^2(dx'dy)}\Big)^{\frac{1}{2}}<\infty\}
$$
with
\[\pd_{x'}^\alpha=\pd_{x_1}^{\alpha_1}\cdots\pd_{x_{d-1}}^{\alpha_{d-1}},\qquad \alpha=(\alpha_1,\cdots,\alpha_{d-1}),\quad|\alpha|=\alpha_1+\cdots+\alpha_{d-1}.\]

Next, we have the following result on the boundedness estimates of the solution to the problem \eqref{pr_linear}.

\begin{prop}\label{prop_est}
 Assume that $\bU_h\in W^{2,\infty}(\R_+)$, the initial data of the problem \eqref{pr_linear} are bounded in the sense that all norms of the initial data appeared in the following estimates are finite, and also satisfy the compatibility conditions of the problem \eqref{pr_linear}. Let $(\bu_h,u_d,\ta)$ be the solution of the problem \eqref{pr_linear},
 then there exist positive constants $M_0=M_0(\|\bU_h(y)\|_{L^\infty(\R_+)})$ and $M_1=M_1(\|\bU_h(y)\|_{W^{2,\infty}(\R_+)})$ independent of $t,$ such that
\begin{equation}\label{est_ta0}
\|\ta(t,\cdot)\|_{L^2(\R^d_+)}\leq\|\ta_0\|_{L^2(\R^d_+)},\quad\|\nabla_h\ta(t,\cdot)\|_{L^2(\R^d_+)} \leq\|\nabla_h\ta_0\|_{L^2(\R^d_+)},
\end{equation}
and
\begin{equation}\label{est_ta1}\begin{split}
&\|\ta\|(t,y)+\|\ta_y\|(t,y)\leq M_0 \big(\|\ta_0\|_{H^{1,0}}+\|\ta_0\|_{H^{0,2}}\big),\\
&\|\nabla_h\ta\|(t,y)\leq M_0\big(\|\ta_0\|_{H^{2,0}}+\|\ta_0\|_{H^{1,2}}\big),\\
&\|\pd_y^2\ta\|(t,y)\leq M_1\big(\|\ta_0\|_{H^{2,0}}+\|\ta_0\|_{H^{1,2}}+\|\ta_0\|_{H^{0,4}}\big)
 \end{split}\end{equation}
 hold for all $t\ge 0$ and $y\ge 0$.
Moreover, one has the following estimates:
 \begin{equation}\label{est_linear}\begin{split}
\|\bu_h\|(t,y)\leq&\|\bu_{h0}\|(y)
+t|\bU_h'(y)|\cdot\int_{0}^{y}\|\nabla_h\cdot\bu_{h0}\|(z)dz+2\|\ta_0\|_{L^{2}(\R^d_+)}\cdot\big|\sqrt{y}\bU_h'(y)\big|,\\
\|u_d\|(t,y)\leq
&\int_{0}^{y}\|\nabla_h\cdot\bu_{h0}\|(z)dz+t\int_{0}^{y}
\Big[\big|\bU_h(y)-\bU_h(z)\big|\cdot\|\nabla_h\big(\nabla_h\cdot\bu_{h0}
\big)\|(z)\Big]dz\\
&+2\|\nabla_h\ta_0\|_{L^2(\R^d_+)}\Big(\int_0^y\big|\bU_h(y)-\bU_h(z)\big|^2dz\Big)^{\frac{1}{2}}+M_0\Big(\|\ta_0\|_{H^{1,0}}+\|\ta_0\|_{H^{0,2}}\Big).
\end{split}\end{equation}
\end{prop}
\begin{proof}[\bf{Proof.}] (1)
Firstly, from Proposition \ref{prop_linear} we know that $\ta(t,x',y)$ satisfies the linear problem \eqref{pr_ta}. Then, it is easy to obtain that by energy estimate,
\begin{equation*}
\frac{d}{2dt}\|\ta(t,\cdot)\|^2_{L^2(\R^d_+)}+\|\pd_y\ta(t,\cdot)\|^2_{L^2(\R^d_+)}~=~0,
\end{equation*}
which implies that
\begin{equation}\label{energy_ta}
\|\ta(t,\cdot)\|_{L^2(\R^d_+)}^2+2\int_{0}^{t}\|\pd_y\ta(s,\cdot)\|_{L^2(\R_+^d)}^2ds~=~\|\ta_0\|_{L^2(\R^d_+)}^2,\quad \forall t\geq0.
\end{equation}
Denote by the operator
\[\pd_{\mathcal{T}}^\alpha~:=~\pd_t^{\alpha_1}\pd_{x_1}^{\alpha_2}
\cdots
\pd_{x_{d-1}}^{\alpha_d},\qquad\alpha=(\alpha_1,\cdots,\alpha_d),\quad|\alpha|=\alpha_1+\cdots+\alpha_d.\]
Applying the operator $\pd_{\mathcal{T}}^\alpha,~|\alpha|=1$ to the equation of \eqref{pr_ta}, and similarly  we have that, 
\begin{equation}\label{energy_x}
\|\pt^\alpha\ta(t,\cdot)\|_{L^2(\R^d_+)}^2+2\int_{0}^{t}\|\pd_y\pt^\alpha\ta(s,\cdot)\|^2_{L^2(\R^d_+)}ds~=~\|\pt^\alpha\ta(0,\cdot)\|_{L^2(\R_+^d)}^2.
\end{equation}


(2)
Combining the equation of \eqref{pr_ta} with the estimate  \eqref{energy_x}, and noting that
\begin{equation}\label{ini_pdta}
\ta_t(0,x',y)=\pd_y^2\ta_0(x',y)-\bU_h(y)\cdot\nabla_h\ta_0(x',y),
\qquad\ta_{x_i}(0,x',y)=\ta_{0x_i}(x',y),\quad 1\leq i\leq d-1,
\end{equation}
 it follows that
\begin{equation}\label{energy_yy}\begin{split}
  \|\pd_y^2\ta(t,\cdot)\|_{L^2(\R^d_+)}&~\leq~\|\ta_t(t,\cdot)\|_{L^2(\R^d_+)}+\|\bU_h(y)\|_{L^\infty(\R_+)}\cdot\|\nabla_h\ta(t,\cdot)\|_{L^2(\R^d_+)}\\
  &~\leq~2\|\bU_h(y)\|_{L^\infty(\R_+)}\cdot\|\nabla_h\ta_0\|_{L^2(\R^d_+)}+\|\pd_y^2\ta_0\|_{L^2(\R^d_+)}.
\end{split}\end{equation}
By the classical interpolation inequality we obtain that from \eqref{energy_ta} and \eqref{energy_yy},
\begin{equation}\label{energy_y}\begin{split}
  \|\ta_y(t,\cdot)\|_{L^2(\R^d_+)}&\leq C\Big(\|\ta(t,\cdot)\|_{L^2(\R^d_+)}+\|\pd_y^2\ta(t,\cdot)\|_{L^2(\R^d_+)}\Big)\\
  &\leq C\Big(\|\ta_0\|_{L^2(\R^d_+)}+\|\bU_h(y)\|_{L^\infty(\R_+)}\cdot\|\nabla_h\ta_0\|_{L^2(\R^d_+)}+\|\pd_y^2\ta_0\|_{L^2(\R^d_+)}\Big),
\end{split}\end{equation}
where $C$ is a positive constant independent of $t$. Then, from the estimates \eqref{energy_ta}, \eqref{energy_yy} and \eqref{energy_y} it implies by the imbedding inequality that there is a positive constant $M_0=M_0(\|\bU_h(y)\|_{L^\infty(\R_+)})$ independent of $t$, such that
\begin{equation}\label{est_tay}\begin{split}
&\|\ta(t,\cdot)\|_{L^{2,\infty}}\leq \|\ta(t,\cdot)\|_{H^{0,1}}\leq M_0\big(\|\ta_0\|_{H^{1,0}}+\|\ta_0\|_{H^{0,2}}\big),\\
&\|\ta_y(t,\cdot)\|_{L^{2,\infty}}\leq \|\ta_y(t,\cdot)\|_{H^{0,1}}\leq M_0\big(\|\ta_0\|_{H^{1,0}}+\|\ta_0\|_{H^{0,2}}\big).
\end{split}\end{equation}

(3) Next, we apply $\pt^\alpha,~|\alpha|=1$ to the equation in \eqref{pr_ta} and get
\begin{equation}\label{eq_pdta}
\pd_t\pt^\alpha\ta+\bU_h(y)\cdot\nabla_h\pt^\alpha\ta-\pd_y^2\pt^\alpha\ta=0,
\end{equation}
moreover, we have the initial data \eqref{ini_pdta} and the following boundary value of $\pt^\alpha\ta(t,x',y)$:
\begin{equation}\label{bd_pdta}
\pt^\alpha\ta(t,x',0)~=~0.
\end{equation}
Thus, by using the same argument as above for the solution $\pt^\alpha\ta$ of the problem \eqref{eq_pdta}-\eqref{bd_pdta}, we can obtain that there exist positive constants $C_1=C_1(\|\bU_h(y)\|_{L^\infty(\R_+)})$ and $C_2=C_2(\|\bU_h(y)\|_{W^{2\infty}(\R_+)})$ independent of $t,$ such that
\begin{equation}\label{est_tax}\begin{split}
&\|\nabla_h\ta(t,\cdot)\|_{L^{2,\infty}}\leq C_1\big(\|\ta_0\|_{H^{2,0}}+\|\ta_0\|_{H^{1,2}}\big),\\
&\|\ta_t(t,\cdot)\|_{L^{2,\infty}}\leq C_2\big(\|\ta_0\|_{H^{2,0}}+\|\ta_0\|_{H^{1,2}}+\|\ta_0\|_{H^{0,4}}\big).
 \end{split}\end{equation}
 Furthermore, from the equation given in \eqref{pr_ta} we obtain that there is a positive constant $C_3=C_3(\|\bU_h(y)\|_{W^{2\infty}(\R_+)})$ independent of $t$, such that
 \begin{equation}\label{est_tayy}\begin{split}
 \|\pd_y^2\ta(t,\cdot)\|_{L^{2,\infty}}&\leq\|\ta_t(t,\cdot)\|_{L^{2,\infty}}+\|\bU_h(y)\|_{L^\infty(\R_+)}\cdot\|\nabla_h\ta\|_{L^{2,\infty}}\\
 &\leq C_3\big(\|\ta_0\|_{H^{2,0}}+\|\ta_0\|_{H^{1,2}}+\|\ta_0\|_{H^{0,4}}\big).
\end{split} \end{equation}
Combining \eqref{est_tay}, \eqref{est_tax} and \eqref{est_tayy}, we obtain the estimates given in \eqref{est_ta1}.

(4) From the representation \eqref{formu-u} of $(\bu_h, u_d)$ given in Proposition \ref{prop_linear}, it is easy to obtain:
\begin{equation}\label{energy_uh}\begin{split}
\|\bu_h\|(t,y)&\leq\|\bu_{h0}\|(y)+|\bU_h'(y)|\int_{0}^{y}\Big[\|\ta_0\|(z)+\|\ta\|(t,z)\Big]dz+t|\bU_h'(y)|\int_{0}^{y}\|\nabla_h\cdot\bu_{h0}\|(z)dz\\
&\leq\|\bu_{h0}\|(y)+t|\bU_h'(y)|\int_{0}^{y}\|\nabla_h\cdot\bu_{h0}\|(z)dz+\big|\sqrt{y}\bU_h'(y)\big|\Big(\|\ta_0\|_{L^2(\R^d_+)}+\|\ta(t,\cdot)\|_{L^2(\R^d_+)}\Big),
\end{split}\end{equation}
and
\begin{equation}\label{energy_ud}\begin{split}
\|u_d\|(t,y)\leq&2\|\ta_y(t,\cdot)\|_{L^{2,\infty}}+\int_{0}^{y}\Big[\big|\bU_h(y)-\bU_h(z)\big|\cdot\Big(\|\nabla_h\ta_0\|(z)+\|\nabla_h\ta\|(t,z)\Big)\Big]dz\\
& +\int_{0}^{y}\|\nabla_h\cdot\bu_{h0}\|(z)dz+t\int_{0}^{y}
\Big[\big|\bU_h(y)-\bU_h(z)\big|\cdot\|\nabla_h\big(\nabla_h\cdot\bu_{h0}
\big)\|(z)\Big]dz\\
\leq&2\|\ta_y(t,\cdot)\|_{L^{2,\infty}}+\Big(\int_{0}^{y}\big|\bU_h(y)-\bU_h(z)\big|^2dz\Big)^{\frac{1}{2}}\cdot\Big(\|\nabla_h\ta_0\|_{L^2(\R^d_+)}+\|\nabla_h\ta(t,\cdot)\|_{L^2(\R^d_+)}\Big)\\
& +\int_{0}^{y}\|\nabla_h\cdot\bu_{h0}\|(z)dz+t\int_{0}^{y}
\Big[\big|\bU_h(y)-\bU_h(z)\big|\cdot\|\nabla_h\big(\nabla_h\cdot\bu_{h0}
\big)\|(z)\Big]dz.
\end{split}\end{equation}
Combining \eqref{energy_ta} with \eqref{energy_uh}, we obtain the estimate of  $\bu_h$ given in \eqref{est_linear}. Substituting \eqref{est_ta1} and \eqref{energy_x} into \eqref{energy_ud}, the estimate of  $u_d$ given in \eqref{est_linear} follows immediately.
\end{proof}

\begin{remark}
	From the
computation given in the above proposition, and  the expressions \eqref{formu-u1} and \eqref{formu-u2} of $(\tilde{\bu}_h,\tilde u_d)$ and $(\bar{\bu}_h,\bar u_d)$ respectively, indeed we can get the following more precise estimates
 for all $t\geq0,y\geq0$, 
	 \begin{equation}\label{est_linear1}\begin{cases}
	 \|\tilde\bu_h\|(t,y)\leq\|\bu_{h0}\|(y)
	 +t|\bU_h'(y)|\cdot\int_{0}^{y}\|\nabla_h\cdot\bu_{h0}\|(z)dz,\\
	 \|\tilde u_d\|(t,y)\leq
	 \int_{0}^{y}\|\nabla_h\cdot\bu_{h0}\|(z)dz+t\int_{0}^{y}
	 \Big[\big|\bU_h(y)-\bU_h(z)\big|\cdot\|\nabla_h\big(\nabla_h\cdot\bu_{h0}
	 \big)\|(z)\Big]dz,
	 \end{cases}\end{equation}
	 and
	  \begin{equation}\label{est_linear2}\begin{cases}
	  \|\bar\bu_h\|(t,y)\leq
	  2\|\ta_0\|_{L^{2}(\R^d_+)}\cdot\big|\sqrt{y}\bU_h'(y)\big|,\\
	  \|\bar u_d\|(t,y)\leq
	  2\|\nabla_h\ta_0\|_{L^2(\R^d_+)}\Big(\int_0^y\big|\bU_h(y)-\bU_h(z)\big|^2dz\Big)^{\frac{1}{2}}+M_0\Big(\|\ta_0\|_{H^{1,0}}+\|\ta_0\|_{H^{0,2}}\Big).
	  \end{cases}\end{equation}
\end{remark}

\subsection{Linearized stability of shear flows in two-dimensional problems}

The next main goal is to improve the estimates given in \eqref{est_linear} to have a lower bound on the growth rate of $(\bu_h, u_d)$ as $t\to+\infty$, under certain structural condition on shear flow  $\bU_h(y)$, which implies the asymptotic instability of the linearized problem \eqref{pr_linear}. Hong and Hunter had studied the similar problem for the linearized two-dimensional inviscid Prandtl equations in \cite{H-H}.

Firstly, we consider the problem \eqref{pr_linear} in two space variables. Note that from \eqref{invis_prandtl},  $(\tilde\bu_h,\tilde u_d)(t,x',y)$ is the solution to the linearized problem of two-dimensional inviscid Prandtl equations, thus from the relation \eqref{decom} and the estimate \eqref{est_linear2} we claim that $\bu_h(t,x',y)$ satisfies similar estimates as given in \cite[Proposition 6.1]{H-H}. Indeed, we have the following result.

\begin{prop}\label{prop_sta2d}
	Under the assumptions of Proposition \ref{prop_est}, let
	$(\bu_h,u_d,\ta)(t,x',y)$ be the solution of \eqref{pr_linear}
in $\{t>0, x'\in \R, y>0\}$.
	
   (1) If $\bU(y)$ has no any critical point, then $\|\bu_h\|(t,y)$ and $\|u_d\|(t,y)$ are bounded uniformly in $t$ with the following estimates:
	\begin{equation}\label{est_uv}\begin{split}
	\|\bu_h\|(t,y)\leq&\Big|\frac{\bU_h'(y)}{\bU_h'(0)}\Big|\|\bu_{h0}\|(0) +|\bU_h'(y)|\int_{0}^{y}\Big\{\Big|\frac{\|\pd_y\bu_{h0}\|(z)}{\bU_h'(z)}\Big|+\Big|\frac{\bU_h''(z)}{\big(\bU_h'(z)\big)^2}\Big|\|\bu_{h0}\|(z)\Big\}dz+2\|\ta_0\|_{L^2(\R^d_+)}|\sqrt{y}\bU_h'(y)|,\\
	\|u_d\|(t,y)\leq&\Big|\frac{\bU_h(y)-\bU_h(0)}{\bU_h'(0)}\Big|\|\pd_{x'}\bu_{h0}\|(0) +\int_{0}^{y}\Big|\frac{\bU_h(y)-\bU_h(z)}{\bU_h'(z)}\Big|\Big[\|\pd_{x'y}^2\bu_{h0}\|(z)+\Big|\frac{\bU_h''(z)}{\bU_h'(z)}\Big|\|\pd_{x'}\bu_{h0}\|(z)\Big]dz\\
	&+2\|\pd_{x'}\ta_0\|_{L^2(\R^d_+)}\Big(\int_0^y\big|\bU_h(y)-\bU_h(z)\big|^2dz\Big)^{\frac{1}{2}}+M_0\Big(\|\ta_0\|_{H^{1,0}}+\|\ta_0\|_{H^{0,2}}\Big).
	\end{split}\end{equation}
	(2) If $\bU(y)$ has a single, non-degenerate critical point at $y=y_0>0$, and the initial data $\bu_{h0}(x',y)$ satisfies
	\begin{equation}
        \|\bu_{h0}\|_{\frac{i}{2}}(y_0)~:=\Big(\int_{\R_\xi}|\xi|^i\cdot|\widehat{\bu_{h0}}(\xi,y_0)|^2d\xi\Big)^{\frac{1}{2}}<\infty, \qquad i=1,2,3,
	\end{equation}
	where $\widehat{\bu_{h0}}(\xi,y)$ is the Fourier transform of $\bu_{h0}(x',y)$ with respect to $x'$,
	 then when $y> y_0$, it holds that 
	 for sufficiently large $t$,
	\begin{equation}\label{est-uv}\begin{split}
	&\|\bu_h\|(t,y)\geq C\sqrt{t}~\frac{|\bU_h'(y)|}{\sqrt{|\bU_h''(y_0)|}},\qquad\|u_d\|(t,y)\geq C\sqrt{t}~\frac{|\bU_h(y)-\bU_h(y_0)|}{\sqrt{|\bU_h''(y_0)|}},
	\end{split}\end{equation}
	where the positive constant $C$ depends only on $y_0$ and $\bu_{h0}$.
	Furthermore, we have similar results as above for $\pd_y\bu_h.$
	\end{prop}
	\begin{proof}[\bf{Proof.}]
		Combining \eqref{decom} with \eqref{est_linear2}, we only need to estimate $\|(\tilde\bu_h,\tilde u_d)\|(t,y)$. As we know, $(\tilde\bu_h,\tilde u_d)(t,x',y)$ solves the linearized inviscid Prandtl equation \eqref{invis_prandtl}, so we can follow the method given in the proof of Proposition 6.1 in \cite{H-H} to have the estimates of $(\tilde\bu_h,\tilde u_d)(t,x',y)$, and  we sketch the process in the following.
		
		  By taking the Fourier transform with respect to $x'\in\R$ in the representation \eqref{formu-u1} of $(\tilde \bu_h,\tilde u_d)(t,x',y)$, it follows that
		\begin{equation}\label{fourier_tu}\begin{split}
		\widehat{\tilde{\bu}_h}(t,\xi,y)=&\widehat{\bu_{h0}}(\xi,y)e^{-it\xi\bU_h(y)}+it\xi\bU_h'(y)\int_{0}^{y}\widehat{{\bu}_{h0}}(\xi,z)e^{-it\xi\bU_h(z)}dz, \\
				\widehat{\tilde u_d}(t,\xi,y)= &-\int_{0}^{y} \Big\{i\xi\widehat{\bu_{h0}}(\xi,y)-t\xi^2\big(\bU_h(y)-\bU_h(z)\big)\widehat{\bu_{h0}}(\xi,z)\Big\}e^{-it\xi\bU_h(z)}dz.
			\end{split}\end{equation}
		
		(1) If $\bU_h(y)$ has no any critical point, we take integration by parts in \eqref{fourier_tu} to obtain that
		\begin{equation*}
		\begin{split}
		\widehat{\tilde{\bu}_h}(t,\xi,y)=&
\frac{\bU'_h(y)}{\bU'_h(0)}
\widehat{\bu_{h0}}(\xi,0)e^{-it\xi\bU_h(0)}+\bU_h'(y)\int_{0}^{y}\big[\frac{\pd_y\widehat{{\bu}_{h0}}(\xi,z)}{\bU_h'(z)}-\frac{\bU_h''(z)}{\big(\bU_h'(z)\big)^2}\widehat{{\bu}_{h0}}(\xi,z)\big]e^{-it\xi\bU_h(z)}dz, \\
		\widehat{\tilde u_d}(t,\xi,y)= &-i\xi\frac{\bU_h(y)-\bU_h(0)}{\bU_h'(0)}\widehat{\bu_{h0}}(\xi,0)e^{-it\xi\bU_h(0)}\\
		&-i\xi\int_{0}^{y} \big[\frac{\pd_y\widehat{{\bu}_{h0}}(\xi,z)}{\bU_h'(z)}
-\frac{\bU_h''(z)}{\big(\bU_h'(z)\big)^2}\widehat{{\bu}_{h0}}(\xi,z)\big]\big(\bU_h(y)-\bU_h(z)\big)
e^{-it\xi\bU_h(z)}dz, 
		\end{split}
		\end{equation*}
which implies by using Parseval's identity,
		\begin{equation}\label{est_uv1}\begin{split}
		\|\tilde{\bu}_h\|(t,y)\leq&
\Big|\frac{\bU'_h(y)}{\bU_h'(0)}\Big|
\|\bu_{h0}\|(0) +|\bU_h'(y)|\int_{0}^{y}\Big[\frac{\|\pd_y\bu_{h0}\|(z)}{|\bU_h'(z)|}+\frac{|\bU_h''(z)|}{\big(\bU_h'(z)\big)^2}\|\bu_{h0}\|(z)\Big]dz,\\
		\|\tilde u_d\|(t,y)\leq&\Big|\frac{\bU_h(y)-\bU_h(0)}{\bU_h'(0)}\Big|\|\pd_{x'}\bu_{h0}\|(0) +\int_{0}^{y}\Big|\frac{\bU_h(y)-\bU_h(z)}{\bU_h'(z)}\Big|\Big[\|\pd_{x'y}^2\bu_{h0}\|(z)+\Big|\frac{\bU_h''(z)}{\bU_h'(z)}\Big|\|\pd_{x'}\bu_{h0}\|(z)\Big]dz.
		\end{split}\end{equation}
		Substituting \eqref{est_linear2} and \eqref{est_uv1} into \eqref{decom}, it follows the estimates given in \eqref{est_uv} immediately.
		
	(2) If $\bU_h(y)$ has a
single,
 non-degenerate critical point at $y=y_0$, through the method of stationary phase we obtain that for $y>y_0$ and as $|t\xi|\rightarrow+\infty,$
	\begin{equation}\label{st-ph}\begin{split}
	&\int_{0}^{y}\widehat{\bu_{h0}}(\xi,z)e^{-it\xi\bU_h(z)}dz=\sqrt{\frac{2\pi}{|t\xi\bU_h''(y_0)|}}\cdot\widehat{\bu_{h0}}(\xi,y_0)e^{-it\xi\bU_h(y_0)-\frac{i\pi}{4}sgn\big(\xi\bU_h''(y_0)\big)}+o(\frac{1}{|t\xi|}),\\
	&\int_{0}^{y}\big[\bU_h(y)-\bU_h(z)\big]\widehat{\bu_{h0}}(\xi,z)e^{-it\xi\bU_h(z)}dz\\
	&=\sqrt{\frac{2\pi}{|t\xi\bU_h''(y_0)|}}\cdot\big[\bU_h(y)-\bU_h(y_0)\big]\widehat{\bu_{h0}}(\xi,y_0)e^{-it\xi\bU_h(y_0)-\frac{i\pi}{4}sgn\big(\xi\bU_h''(y_0)\big)}+o(\frac{1}{|t\xi|}).
	\end{split}	\end{equation}
	Then, combining \eqref{fourier_tu} with \eqref{st-ph} yields that for $\xi$ being in a bounded interval,  $\xi\in[a,b]$ with $0<a<b$, the following inequalities hold for sufficiently large $t$ (independent of $\xi$),
	\begin{equation}\label{as-ex}\begin{split}
	&\big|\widehat{\tilde{\bu}_h}(t,\xi,y)-\widehat{\bu_{h0}}(\xi,y)e^{-it\xi\bU_h(y)}\big|\geq\sqrt{\frac{\pi|t\xi|}{|\bU_h''(y_0)|}}\cdot\big|\bU_h'(y)\widehat{\bu_{h0}}(\xi,y_0)\big|,\\
	&\big|\widehat{\tilde u_d}(t,\xi,y)\big|\geq\sqrt{\frac{\pi|t\xi|}{|\bU_h''(y_0)|}}\cdot\big|\xi\big[\bU_h(y)-\bU_h(y_0)\big]\widehat{\bu_{h0}}(\xi,y_0)\big|.
	\end{split}	\end{equation}
	Therefore, by Parseval's identity we obtain that for sufficiently large $t$,
	\begin{equation}\label{est-uv1}\begin{split}
		&\|\tilde \bu_h\|(t,y)\geq\sqrt{\frac{\pi t}{2|\bU_h''(y_0)|}}|\bU_h'(y)|\cdot\big\|\sqrt{|\xi|}\widehat{\bu_{h0}}(\xi,y_0)\big\|_{L^2_\xi([a,b])}\triangleq C_0 \sqrt{t}~\frac{|\bU_h'(y)|}{\sqrt{|\bU_h''(y_0)|}},\\
		&\|\tilde u_d\|(t,y)\geq\sqrt{\frac{\pi t}{2|\bU_h''(y_0)|}}\big|\bU_h(y)-\bU_h(y_0)\big|\cdot\Big\||\xi|^{\frac{3}{2}}\widehat{\bu_{h0}}(\xi,y_0)\Big\|_{L^2_\xi([a,b])}\triangleq C_1 \sqrt{t}~\frac{|\bU_h(y)-\bU_h(y_0)|}{\sqrt{|\bU_h''(y_0)|}},\\
	\end{split}\end{equation}
	where the positive constants $C_0$ and $C_1$ depend only on $y_0$ and $\bu_{h0}$. Finally, it is easy to obtain \eqref{est-uv} by substituting \eqref{est_linear2} and \eqref{est-uv1} into \eqref{decom}.
	\end{proof}

\subsection{Linearized stability of shear flows in three-dimensional problems}
We shall see that
the results on linear stability of shear flows in three-dimensional case are different from the ones in two-dimensional case given in the above subsection. By using the decomposition \eqref{decom} and the estimates of $(\bar{\bu}_h,\bar u_d)(t,x',y)$ given in  \eqref{est_linear2}, we will focus on the component $(\tilde{\bu}_h,\tilde u_d)(t,x',y)$ which satisfies the three-dimensional inviscid linearized Prandtl system \eqref{invis_prandtl}.
In analogy with
the well-posedness result and ill-posedness result  on the three-dimensional viscous Prandtl system
given in \cite{LWY1} and \cite{LWY3} respectively, we
will deduce that the structure of the shear flow $\bU_h(y)$ plays an important role on its linear stability. In fact, we have the following result:

\begin{prop}\label{thm_sta} Consider the linearized problem \eqref{pr_linear} in three space variables,
suppose that $\bU_h(y)=\big(U_1(y),U_2(y)\big)$ and the initial data are smooth, and
the norms appeared on the right hand side of the following estimate
\eqref{est_ud1} are finite.
 Let $(\bu_h,u_d,\ta)(t,x',y)$ be the solution of \eqref{pr_linear}.

(1)
If there is $k\in \R$ such that $U_2(y)=kU_1(y)$ holds for all $y\ge 0$, and
$U_1(y)$ has no critical point in $y\ge 0$,
then $\|u_d\|(t,y)$ is bounded uniformly in $t$, and satisfies the estimate:
\begin{equation}\label{est_ud1}\begin{split}
\|u_d\|(t,y)\leq&\Big|\frac{U_1(y)-U_1(0)}{U_1'(0)}\Big|\cdot\|\nabla_h\cdot\bu_{h0}\|(0) +\int_{0}^{y}\Big\{\Big|\frac{U_1(y)-U_1(z)}{U_1'(z)}\Big|\cdot\|\nabla_h\cdot\pd_y\bu_{h0}\|(z)\\
&\qquad+\Big|\frac{U_1''(z)\big(U_1(y)-U_1(z)\big)}{\big(U_1'(z)\big)^2}\Big|\cdot\|\nabla_h\cdot\bu_{h0}\|(z)\Big\}dz\\
&+2\|\nabla_h\ta_0\|_{L^2(\R^d_+)}\Big(\int_0^y\big|\bU_h(y)-\bU_h(z)\big|^2dz\Big)^{\frac{1}{2}}+M_0\Big(\|\ta_0\|_{H^{1,0}}+\|\ta_0\|_{H^{0,2}}\Big).
\end{split}\end{equation}

(2)  Assume that the initial data of \eqref{pr_linear} admits
\[
\|\nabla_h\cdot\bu_{h0}\|_{\frac{i}{2}}(y)~:=~\Big(\int_{\R^2}|\xi|^i\cdot\big|\xi\cdot\widehat{\bu_{h0}}\big|^2(\xi,y)d\xi\Big)^{\frac{1}{2}}<+\infty,\quad i=0,1,
\]
where $\widehat{\bu_{h0}}(\xi, y)$ denotes the Fourier transform of ${\bu}_{h0}(x_1,x_2, y)$ with respect to
$(x_1,x_2)$.

(2a) If there is $k\in \R$ such that $U_2(y)=kU_1(y)$ holds for all $y\ge 0$, and $U_1(y)$ has a single, non-degenerate critical point at $y=y_0>0$, then when $y> y_0$, there exists a constant $C=C\big(y,y_0,\bU_h,\bu_{h0}\big)>0$ independent of $t$, such that for sufficiently large $t$,
\begin{equation}\label{est_ud4}
\|u_d\|(t,y)~\geq~
C\sqrt{t}~\frac{|U_1(y)-U_1(y_0)|}{\sqrt{|U_1''(y_0)|}} .
\end{equation}

(2b) If for any given $k\in \R$, $U_2(y)=kU_1(y)$ does not hold for all $y\ge 0$, then there is a point $y_0> 0$ such that, when $y>y_0$ we have that for sufficiently large $t,$
\begin{equation}\label{est_ud5}\begin{split}
\|u_d\|(t,y)~\geq~C\sqrt{t}
\end{split}\end{equation}
with the constant $C=C\big(y,y_0,\bU_h,\bu_{h0}\big)>0$ independent of $t.$ Moreover, we have similar results as above for $\pd_yu_d$ and $\nabla_h\cdot\bu_h.$
\end{prop}

\begin{proof}[\bf{Proof.}]
	As in the proof of Proposition \ref{prop_sta2d},
	by using \eqref{decom} and \eqref{est_linear2} we only need to study $(\tilde\bu_h,\tilde u_d)(t,x',y)$, which solves the three-dimensional linearized inviscid Prandtl equations \eqref{invis_prandtl}.
	
	By taking the Fourier transform with respect to $x'=(x_1, x_2)^T$ in the representation \eqref{formu-u1} of $\tilde u_d(t,x',y)$, we obtain that
	\begin{equation}\label{fourier_u}\begin{split}
	\widehat{\tilde u_d}(t,\xi,y)= &-\int_{0}^{y} \Big\{i\xi\cdot\widehat{\bu_{h0}}(\xi,y)-t\big[\xi\cdot\big(\bU_h(y)-\bU_h(z)\big)\big]\cdot\big[\xi\cdot\widehat{\bu_{h0}}(\xi,z)
	\big]\Big\}e^{-it\xi\cdot\bU_h(z)}dz.
	\end{split}\end{equation}
	
 (1)  When $U_2(y)=kU_1(y)$ for some constant $k\in\R$, then \eqref{fourier_u} is reduced as
  \begin{equation}\label{four}\begin{split}
  &\widehat{\tilde u_d}(t,\xi,y)=-\int_{0}^{y}\big[1-t(\xi_1+k\xi_2)\big(U_1(y)-U_1(z)\big)\big]\big[ \xi\cdot\widehat{\bu_{h0}}(\xi,z)\big]e^{-it(\xi_1+k\xi_2)U_1(z)}dz.
  \end{split}\end{equation}

If $U_1(y)$ has no any critical point for all $y\ge 0$, then when  $\xi_1+k\xi_2\neq 0$, we obtain that by integration by parts,
\begin{equation}\label{four_1} \begin{split}
  \widehat{\tilde u_d}(t,\xi,y)=&-\frac{U_1(y)-U_1(0)}{U_1'(0)}\big[i \xi\cdot\widehat{\bu_{h0}}(\xi,0)\big]e^{-it(\xi_1+k\xi_2)U_1(0)} -\int_{0}^{y}\Big\{\frac{U_1(y)-U_1(z)}{U_1'(z)}\big[i \xi\cdot\pd_y\widehat{\bu_{h0}}(\xi,z)\big]\\
  &\qquad\qquad-\frac{U_1''(z)\big(U_1(y)-U_1(z)\big)}{\big(U_1'(z)\big)^2}\big[i \xi\cdot\widehat{\bu_{h0}}(\xi,z)\big]\Big\}e^{-it(\xi_1+k\xi_2)U_1(z)}dz.
\end{split}\end{equation}
which implies that by Parseval's identity,
\begin{equation}\label{estimate_ud}\begin{split}
\|\tilde u_d\|(t,y)\leq&\Big|\frac{U_1(y)-U_1(0)}{U_1'(0)}\Big|\cdot\|\nabla_h\cdot\bu_{h0}\|(0) +\int_{0}^{y}\Big\{\Big|\frac{U_1(y)-U_1(z)}{U_1'(z)}\Big|\cdot\|\nabla_h\cdot\pd_y\bu_{h0}\|(z)\\
&\qquad+\Big|\frac{U_1''(z)\big(U_1(y)-U_1(z)\big)}{\big(U_1'(z)\big)^2}\Big|\cdot\|\nabla_h\cdot\bu_{h0}\|(z)\Big\}dz.
\end{split}\end{equation}
Thus, by plugging \eqref{est_linear2} and the above estimate \eqref{estimate_ud} into \eqref{decom} we conclude \eqref{est_ud1}.

(2)
If $U_2(y)=kU_1(y)$ for some constant $k$ and $U_1(y)$ has a non-degenerate critical point at $y=y_0$, then for $y>y_0$ and $\xi_1\neq-k\xi_2$, by the method of stationary phase it yields that
as $|t(\xi_1+k\xi_2)|\rightarrow+\infty$,
\begin{equation}\label{mop}\begin{split}
&\int_{0}^{y}\big(U_1(y)-U_1(z)\big)\big[\xi\cdot\widehat{\bu_{h0}}(\xi,z)\big]e^{-it(\xi_1+k\xi_2)U_1(z)}dz\\
&=\sqrt{\frac{2\pi}{|t(\xi_1+k\xi_2)U''_1(y_0)|}}~\big(U_1(y)-U_1(y_0)\big)\big[\xi\cdot\widehat{\bu_{h0}}(\xi,y_0)\big]\exp\big\{-it(\xi_1+k\xi_2)U_1(y_0)-\frac{i\pi}{4}sgn\big((\xi_1+k\xi_2)U_1''(y_0)\big)\big\}\\
&\quad+o(\frac{1}{|t(\xi_1+k\xi_2)|}).
\end{split}\end{equation}
Substituting the above estimate into \eqref{four} we obtain that for  $\xi\in S$, with a bounded domain $S\subset\R^2$ being away from the line $\{\xi|~ \xi_1+k\xi_2=0\}$,
\begin{equation}\label{estimate_ud1}
	\|\widehat{\tilde u_d}\|(t,y)~\geq~\sqrt{\frac{\pi|\xi_1+k\xi_2| t}{|U''_1(y_0)|}}\cdot\Big|\big(U_1(y)-U_1(y_0)\big)\big[\xi\cdot\widehat{\bu_{h0}}(\xi,y_0)\big]\Big|,\qquad \mbox{for~sufficiently~large}~t,
\end{equation}
which implies by using Parseval's identity,
\begin{equation}\label{estimate_ud2}
\|\tilde u_d\|(t,y)~\geq~\sqrt{\pi t}~\frac{|U_1(y)-U_1(y_0)|}{\sqrt{U''_1(y_0)}}~
\|\sqrt{\xi_1+k\xi_2}~\xi\cdot\widehat{\bu_{h0}}(\xi,y_0)
\|_{L^2_\xi(S)}
\end{equation}
for $y>y_0$ and $t$ large.
Then, from the uniform boundedness of $\|\bar u_d\|(t,y)$  with respect to $t$ obtained in \eqref{est_linear2}, we deduce \eqref{est_ud4} immediately.

(3)
If for any given $k\in \R$, $U_2(y)=kU_1(y)$ does not hold for all $y\ge 0$, 
and both of $U_1(y), U_2(y)$ vanish at infinity, then by a contradiction argument, one can show that
there is a point $y_0$ such that
\begin{equation}\label{condition}
U_1'(y_0)U_2''(y_0)~\neq~U_2'(y_0)U_1''(y_0).
\end{equation}
Without loss of generality, we may assume that $U_1'(y_0)>0$ and $U_1'(y_0)U_2''(y_0)-U_2'(y_0)U_1''(y_0)>0.$ Then, we affirm that for any $\delta>0$, there is an interval $S_\delta\subseteq(y_0-\delta,y_0+\delta)$ such that
\begin{equation}\label{def_S}
U_1'(y)>0,\quad U_2'(y)\neq0,\quad U_1'(y)U_2''(y)-U_2'(y)U_1''(y)>0,\qquad\forall~y\in S_\delta,
\end{equation}
which implies that the function $\frac{U_2'(y)}{U_1'(y)}$ is monotonically increasing in $S_\delta.$

Denote by
\begin{equation}\label{def_I}
I_\delta^R~:=~\Big\{\xi=(\xi_1,\xi_2)\in\R^2\setminus\{0\}:~|\xi|\leq R,~\mbox{and}~\exists y\in S_\delta,~s.t.~~\xi\cdot\bU_h'(y)=0\Big\}.
\end{equation}
From the monotonicity of $\frac{U_2'(y)}{U_1'(y)}$ in $S_\delta$, we know that
	 for fixed $\xi\in I_\delta^R$, there is only one point $y\in S_\delta$ satisfying $\xi\cdot\bU_h'(y)=0$.
 Moreover, by virtue of the continuity of $\bU_h'(y)$, it is easy to know that the Lebesgue measure of $I_\delta^R$ is positive, i.e., $m (I_\delta^R)>0.$
Thus, when $y>y_0$ and for any $\xi\in I_\delta^R$ with $\delta\leq y-y_0$, we have $S_\delta\subseteq(0,y),$ and there exists a unique $y_\xi\in S_\delta$ such that $\xi\cdot\bU_h'(y_\xi)=0$ and $\xi\cdot\bU_h''(y_\xi)\neq0$ by using \eqref{def_S}.  For such $(\xi,y)$ it yields that by the method of stationary phase, as $t\rightarrow+\infty$,
\begin{equation}\label{estimate_ud3}\begin{split}
	&\int_{0}^{y}\big[\xi\cdot\big(\bU_h(y)-\bU_h(z)\big)\big]\cdot\big[\xi\cdot\widehat{\bu_{h0}}(\xi,z)\big]~e^{-it\xi\cdot\bU_h(z)}dz\\
	&=\sqrt{\frac{2\pi}{t\big|\xi\cdot\bU_h''(y_\xi)\big|}}~\big[\xi\cdot\big(\bU_h(y)-\bU_h(y_\xi)\big)\big]\cdot\big[\xi\cdot\widehat{\bu_{h0}}(\xi,y_\xi)\big]~e^{-it\xi\cdot\bU_h(y_\xi)-\frac{i\pi}{4}sgn\big(\xi\cdot\bU''_h(y_\xi)\big)}+o(\frac{1}{t}).
\end{split}\end{equation}
Note that when $\delta$ is small enough, we have that for any $\xi\in I_\delta^R,$
\[
\Big|\frac{\xi\cdot\big(\bU_h(y)-\bU_h(y_\xi)\big)}{\sqrt{\big|\xi\cdot\bU_h''(y_\xi)\big|}}\big[\xi\cdot\widehat{\bu_{h0}}(\xi,y_\xi)\big]\Big|~\geq~\Big|\frac{\xi\cdot\big(\bU_h(y)-\bU_h(y_0)\big)}{2\sqrt{\big|\xi\cdot\bU_h''(y_0)\big|}}\big[\xi\cdot\widehat{\bu_{h0}}(\xi,y_0)\big]\Big|,
\]
and then, substituting \eqref{estimate_ud3} into \eqref{fourier_u} implies that for $\xi\in I_\delta^R$ and $t$ large enough,
\begin{equation}\label{mop1}\begin{split}
\big|\widehat{\tilde u_d}(t,\xi,y)\big|~\geq~&
\frac{\sqrt{\pi t}}{2}~\Big|\frac{\xi\cdot\big(\bU_h(y)-\bU_h(y_0)\big)}{\sqrt{\big|\xi\cdot\bU_h''(y_0)\big|}}\big[\xi\cdot\widehat{\bu_{h0}}(\xi,y_0)\big]\Big|.
\end{split}\end{equation}
Thus, for sufficiently large $t$ we obtain that by using Parseval's identity in \eqref{mop1},
\[\begin{split}
&\|\tilde u_d\|(t,y)\geq\frac{\sqrt{\pi t}}{2}~\Big\|\frac{\xi\cdot\big(\bU_h(y)-\bU_h(y_0)\big)}{\sqrt{\big|\xi\cdot\bU_h''(y_0)\big|}}\big[\xi\cdot\widehat{\bu_{h0}}(\xi,y_0)\big]\Big\|_{L^2_\xi(I_\delta^R)},
\end{split}\]
and then, combining with the uniform boundedness of $\|\bar u_d\|(t,y)$ given in \eqref{est_linear2},
it implies that,
\[
\|u_d\|(t,y)\geq\frac{\sqrt{\pi t}}{4}~\Big\|\frac{\xi\cdot\big(\bU_h(y)-\bU_h(y_0)\big)}{\sqrt{\big|\xi\cdot\bU_h''(y_0)\big|}}\big[\xi\cdot\widehat{\bu_{h0}}(\xi,y_0)\big]\Big\|_{L^2_\xi(I_\delta^R)}.
\]
Consequently, we get the estimate \eqref{est_ud5}. Through analogous arguments as above, we can obtain similar results for $\partial_yu_d$ and $\nabla_h\cdot\bu_h$.

\end{proof}

\begin{remark}\label{rem_sta}
 From Proposition \ref{thm_sta}, we see that when the velocity field direction of the background shear flow $(U_1(y), U_2(y), 0)$ is invariant in the normal variable, the linearized problem  \eqref{pr_linear} is asymptotically stable when the tangential velocity $U_1(y)$ is monotonic, and unstable when it has a non-degenerate critical point,
on the other hand, when the velocity field direction of $(U_1(y), U_2(y), 0)$ changes with respect to $y$, then
the problem  \eqref{pr_linear}
is always asymptotically unstable.
 This interesting phenomenon is analogy to the stability and instability results
obtained by authors in \cite{LWY1, LWY2, LWY3}
for the three dimensional incompressible Prandtl equations.

\end{remark}

\section*{Appendix. Derivation of the boundary layer problem}
\setcounter{equation}{0}
\renewcommand{\theequation}{A.\arabic{equation}}



In the appendix, we  give a formal derivation of the problem \eqref{pr_invis} for the thermal layer profiles in the zero heat conductivity limit of inviscid compressible flows. Analogously,
this problem of the thermal layer profiles can also be derived from the compressible Navier
Stokes equations when the viscosity coefficients are of higher order with respect to the heat conductivity coefficient.

  Consider the following problem of the compressible Euler-Fourier equations in the domain $\R_+\times\R^d_+$ 
with $d=2,3,$
\begin{equation}\label{oreq}
\begin{cases}
\partial_t\rho+\nabla\cdot(\rho \textbf{u})=0,\\
\rho\{
\partial_t\textbf{u}+(\textbf{u}\cdot\nabla)\textbf{u}\}+\nabla
p(\rho,\theta) =0,\\
c_V\rho\{\partial_t\theta+(\textbf{u}\cdot\nabla)\theta\}+p(\rho,\theta)\nabla\cdot \textbf{u} =\ep\Delta\ta,
\end{cases}\end{equation}
where the spatial variables $x=(x',x_d)\in\R_+^d$ with $x'=(x_1,\cdots,x_{d-1})\in\R^{d-1}$ and $x_d>0$, $\rho$ is the density,  $\textbf{u}=(u_1,\cdots,u_d)^T$ is the velocity, $\theta$ is the absolute temperature, $p(\rho,\theta)$ is the pressure,
the constant $c_V>0$ is the specific heat capacity,
$\ep$ is the coefficient of heat conduction.
For the equations \eqref{oreq}, we endow them with the following 
boundary conditions:
\begin{equation}\label{bd_oreq}
u_d|_{x_d=0}=0,\quad \big[\alpha\pd_{x_d}\ta+\beta\ta\big]\big|_{x_d=0}=\gamma,
\end{equation}
where 
$\alpha=\alpha(t,x'),\beta=\beta(t,x')$ and $\gamma=\gamma(t,x')$ are given functions. 
For simplicity, we consider the ideal gas model for the problem \eqref{oreq}-\eqref{bd_oreq}, i.e.,
\(p(\rho,\theta)=R\rho\theta\)
with a positive constant $R$. We are concerned with the asymptotic behavior of the solution $(\rho,\textbf{u},\ta)(t,x)$ to the problem \eqref{oreq}-\eqref{bd_oreq}
when 
the heat conduction coefficient $\ep$ tends to zero.

Formally, when $\epz$, the equations \eqref{oreq} goes to the following
compressible non-isentropic Euler equations in $\R_+\times\R^d_+:$
\begin{equation}\label{eq_e}
\begin{cases}
\partial_t\rho^e+\nabla\cdot(\rho^e \textbf{u}^e)=0,\\
\rho^e\{
\partial_t\textbf{u}^e+(\textbf{u}^e\cdot\nabla)\textbf{u}^e\}+R\nabla
(\rho^e\theta^e) =0,\\
c_V\rho^e\{\partial_t\theta^e+(\textbf{u}^e\cdot\nabla)\theta^e\}+R\rho^e\theta^e(\nabla\cdot \textbf{u}^e)=0.
\end{cases}\end{equation}
From the impenetrable condition given in \eqref{bd_oreq}, we know that
the condition \begin{equation}\label{bd_e}
u_d^e|_{x_d=0}=0,
\end{equation}
is a reasonable one to determine the flow described by \eqref{eq_e}, as the flow moves sliply on the boundary $\{x_d=0\}$. In particular, we do not impose any constrain of temperature on the boundary.

The inconsistent of boundary conditions between \eqref{bd_oreq} and \eqref{bd_e} leads to the appearance
of boundary layers near the physical boundary $\{x_d=0\}$, in which the termperture shall change rapidly. Since in the problem \eqref{oreq}-\eqref{bd_oreq}, the heat diffusion is important in the boundary layer and should be balanced by the convection, meanwhile note that the vertical component of the velocity field vanishes at the boundary, then as in \cite{prandtl, LWY4, Gerard, Gerard2}, the
boundary layer is of the characteristic type, and the size of boundary layer is of order $\mathcal{O}(\sep)$.
Therefore, we assume that near the boundary, the solution of \eqref{oreq}-\eqref{bd_oreq} has the form of
\begin{equation}\label{form}
(\rho,\bu,\ta)(t,x)=\Big(\rho^\ep,\bu_h^\ep,\sep u_d^\ep,\ta^\ep\Big)(t,x',\frac{x_d}{\sep})
\end{equation}
with
\(\bu_h^\ep=(u_1^\ep,\cdots,u_{d-1}^\ep)^T\). In these new variables, the problem \eqref{oreq}-\eqref{bd_oreq} is transformed into the following one in
$\{(t,x',y): t>0,x'\in\R^{d-1},y>0\}$ with $y=\frac{x_d}{\sep}$:
\begin{equation}\label{pr_ep}
\begin{cases}
\partial_t\rho^\ep+\nabla_h\cdot(\rho^\ep \bu_h^\ep)+\pd_y(\rho^\ep u_d^\ep)=0,\\
\rho^\ep\{
\partial_t\textbf{u}^\ep_h+(\textbf{u}^\ep_h\cdot\nabla_h+u_d^\ep\pd_y)\textbf{u}^\ep_h\}+R\nabla_h
(\rho^\ep \theta^\ep) =0,\\
\rho^\ep\{
\partial_t u_d^\ep+(\textbf{u}^\ep_h\cdot\nabla_h+u_d^\ep\pd_y)u_d^\ep\}+\frac{R\pd_y
	(\rho^\ep \theta^\ep) }{\ep}=0,\\
c_V\rho^\ep\{\partial_t\theta^\ep+(\textbf{u}^\ep_h\cdot\nabla_h+u_d^\ep\pd_y)\theta^\ep\}+R\rho^\ep\theta^\ep(\nabla_h\cdot \textbf{u}_h^\ep+\pd_yu_d^\ep) =\ep\Delta_h\ta^\ep+\pd_y^2\ta^\ep,\\
u_d^\ep|_{y=0}=0,~\big[\frac{\alpha}{\sep}\pd_y\ta^\ep+ \beta\ta^\ep\big]|_{y=0}=\gamma,
\end{cases}\end{equation}
where
$\nabla_h=(\pd_{x_1},\cdots,\pd_{x_{d-1}})^T$, $\Delta_h=\pd_{x_1}^2+\cdots+\pd_{x_{d-1}}^2$.


Inspired by the Prandtl boundary layer theory of  incompressible flows given in \cite{prandtl},
we assume that the solution of \eqref{pr_ep} can be approximated as follows:
\begin{equation}\label{ansatz}
(\rho^\ep,\bu_h^\ep,u_d^\ep,\ta^\ep)(t,x',y)=(\rho^e,\bu_h^e,\frac{u_d^e}{\sep},\ta^e)(t,x',\sep y)+(\rho^b,\bu_h^b,u_d^b,\ta^b)(t,x',y)+O(\sep),
\end{equation}
where $(\rho^e,\bu^e,\ta^e)$ denotes the Euler flow given by \eqref{eq_e}-\eqref{bd_e} with $\bu^e=(\bu_h^e,u_d^e)^T,$ and the boundary layer profiles $(\rho^b,\bu_h^b,u_d^b,\ta^b)(t,x',y)$ decrease rapidly as $y\rightarrow+\infty$.

Obviously, from \eqref{ansatz} we have
	\begin{equation}\label{ansatz1} (\rho^\ep,\bu_h^\ep,u_d^\ep,\ta^\ep)(t,x',y)=(\rho,\bu_h,u_d,\ta)(t,x',y)+O(\sep).
	\end{equation}
where
\[(\rho,\bu_h,u_d,\ta)(t,x',y)~:=~(\rho^e,\bu_h^e,y\pd_{x_d}u_d^e,\ta^e)(t,x',0)
+(\rho^b,\bu_h^b,u_d^b,\ta^b)(t,x',y)\]
are the boundary layer profiles.

Plugging the ansatz \eqref{ansatz1} into the problem \eqref{pr_ep} and collecting the leading terms in $\ep$,
we obtain the following problem in $\{(t,x',y)|~ t>0, x'\in \R^{d-1}, y>0\}$:
\begin{equation}\label{pr_p}
\begin{cases}
\partial_t\rho+\nabla_h\cdot(\rho \bu_h)+\pd_y(\rho u_d)=0,\\
\rho\{
\partial_t\textbf{u}_h+(\textbf{u}_h\cdot\nabla_h+u_d\pd_y)\textbf{u}_h\}+R\nabla_h
(\rho \theta) =0,\\
\pd_y(\rho\ta)=0,\\
c_V\rho\{\partial_t\theta+(\textbf{u}_h\cdot\nabla_h+u_d\pd_y)\theta\}+R\rho\theta(\nabla_h\cdot \textbf{u}_h+\pd_yu_d) =\pd_y^2\ta,\\
u_d|_{y=0}=0,\quad\liy(\rho,\bu_h,\ta)=(\rho^e,\bu_h^e,\ta^e)(t,x',0),
\end{cases}\end{equation}
and the boundary values for $\ta:$
\begin{equation}\label{bd_ta}\begin{cases}
\pd_y\ta|_{y=0}=0,\qquad&{\rm when}\quad \alpha\neq0,\\
\ta|_{y=0}=\ta^0(t,x'),\qquad&{\rm when}\quad \alpha=0,
\end{cases}\end{equation}
with $\ta^0(t,x'):=\frac{\gamma(t,x')}{\beta(t,x')}$ provided $\beta\neq0.$

Firstly, we immediately obtain that from the third equation and boundary conditions given in \eqref{pr_p},
\begin{equation}\label{bd_pre}
(\rho\ta)(t,x',y)~\equiv~(\rho^e\ta^e)(t,x',0)~=~\frac{p^e(t,x',0)}{R},
\end{equation}
where $p^e$ is the pressure of the Euler flow. It means that the leading term of the pressure does not change in boundary layers.
Next, for the problem \eqref{pr_p} endowed with the Neumann boundary condition for $\ta$ given in \eqref{bd_ta}, i.e., $\alpha\neq0, \pd_y\ta|_{y=0}=0,$ one can check that
\[
(\rho,\bu_h,\ta)(t,x',y)~=~(\rho^e,\bu_h^e,\ta^e)(t,x',0),\quad u_d(t,x',y)=y\pd_{x_d}u_d^e(t,x',0)
\]
is a special solution to \eqref{pr_p}.
Indeed, it can be easily verified by restricting the equations \eqref{eq_e} to the boundary $\{x_d=0\}$ and using the boundary condition \eqref{bd_e}. This shows that when $\alpha\neq 0$ in the boundary condition \eqref{bd_oreq}, the leading term of boundary layer profiles does not appear and the thermal layer for  the compressible system \eqref{oreq} shall be `weak'.  Usually, it is not true when we use the Dirichlet boundary condition $\ta|_{y=0}=\theta^0(t,x')$ given in \eqref{bd_ta} for the problem \eqref{pr_p}, that is to say, the thermal layer for the system \eqref{oreq} is `strong' when it is endowed with the boundary condition \eqref{bd_oreq} with $\alpha=0$.
Therefore, we focus on the case of $\alpha=0$ in the following.

 Plugging \eqref{bd_pre} into the problem \eqref{pr_p}-\eqref{bd_ta}, it follows that $(\bu_h,u_d,\ta)(t,x',y)$
satisfies the following problem in $\R_+\times\R^d_+$:
\begin{equation}\label{pr_p1}\begin{cases}
\pd_t \bu_h+(\bu_h\cdot\nabla_h+u_d\pd_y)\bu_h+\frac{R\ta}{P}\nabla_hP=0,\\
\pd_t \ta+(\bu_h\cdot\nabla_h+u_d\pd_y)\ta
=\frac{R}{(R+c_V)P}\ta\big( \pd_y^2\ta+\bu_h\cdot\nabla_hP+P_t\big),\\
\nabla_h\cdot\bu_h+\pd_y u_d=\frac{ R}{(R+c_V)P} \pd_y^2\ta-\frac{c_V}{(R+c_V)P}\big(\bu_h\cdot\nabla_hP+ P_t\big),\\
(u_d,\ta)|_{y=0}=\big(0,\ta^0(t,x')\big),
\quad
\lim\limits_{\yinf}(\bu_h,\ta)=(\textbf{U}_h,\Ta)(t,x),
\end{cases}\end{equation}
where
$$(P,\textbf{U}_h,\Ta)(t,x')~=~(p^e,\bu_h^e,\ta^e)(t,x',0)$$
are given by the Euler flow, and satisfy the following equations derived from  \eqref{eq_e}-\eqref{bd_e},
\begin{equation}\label{Ber}\begin{cases}
\pd_t\textbf{U}_h+\textbf{U}_h\cdot\nabla_h\textbf{U}_h+\frac{R\Ta}{P}\nabla_hP=0,\\
\pd_t\Ta+\textbf{U}_h\cdot\nabla_h\Ta-\frac{R\Ta}{(R+c_V)P}\cdot(P_t+\textbf{U}_h\cdot\nabla_hP)=0.
\end{cases}\end{equation}
Then, we endow the problem \eqref{pr_p1} with the initial data
\begin{equation}\label{initial}
(\bu_h,\ta)(0,x',y)~=~(\bu_{h0},\ta_0)(x',y).
\end{equation}
If the initial data $\bu_{h0}$ satisfies the compatibility condition
\begin{equation}\label{ass_init}
\liy \bu_{h0}~=~\textbf{U}_h(0,x'),
\end{equation}
 we observe that the constrain of $\bu_h$ as $y\rightarrow+\infty$ in \eqref{pr_p1} can be removed, since the condition $\liy\bu_h=\textbf{U}_h(t,x')$ holds trivially from $\eqref{Ber}_1$ and \eqref{ass_init}, provided that $\bu_h$ has a limit when $y\to +\infty$. Therefore,
we conclude the following initial-boundary value problem for the inviscid Prandtl equations coupled with a degenerate parabolic equation in $\R_+\times\R^d_+$:
\begin{equation}\label{pr_bd1}\begin{cases}
\pd_t \bu_h+(\bu_h\cdot\nabla_h+u_d\pd_y)\bu_h+\frac{R\ta}{P}\nabla_hP=0,\\
\pd_t \ta+(\bu_h\cdot\nabla_h+u_d\pd_y)\ta
=\frac{\ka\ta}{P}\big( \pd_y^2\ta+\bu_h\cdot \nabla_hP+P_t\big),\\
\nabla_h\cdot\bu_h+\pd_y u_d=\frac{\ka}{P} \pd_y^2\ta-\frac{1-\ka}{P}\big(\bu_h\cdot \nabla_hP+ P_t\big),\\
(u_d,\ta)|_{y=0}=\big(0,\ta^0(t,x')\big),
\quad\lim\limits_{\yinf}\ta(t,x,y)=\Ta(t,x'),\\
(\bu_h,\ta)|_{t=0}=(\bu_{h0},\ta_0)(x',y)
\end{cases}\end{equation}
with 
the constant
\(\ka:=\frac{R}{R+c_V}.\)
Finally, we point out that the theoretic study developed in previous sections is focused on a simple case of the problem \eqref{pr_bd1}, i.e., the pressure $P(t,x')$ of the outflow is a positive function depending only on the time variable $t$, 
\[P(t,x')~\equiv~P(t)>0,\]
and thus, the problem  \eqref{pr_bd1} is simplified as the following one,
\begin{equation}\label{pr_bd2}\begin{cases}
\pd_t \bu_h+(\bu_h\cdot\nabla_h+u_d\pd_y)\bu_h=0,\\
\pd_t \ta+(\bu_h\cdot\nabla_h+u_d\pd_y)\ta
=\frac{\ka}{P}\ta \pd_y^2\ta+\frac{\ka P_t}{P}\ta,\\
\nabla_h\cdot\bu_h+\pd_y u_d=\frac{\ka}{P} \pd_y^2\ta-\frac{(1-\ka)P_t}{P},\\
(u_d,\ta)|_{y=0}=\big(0,\ta^0(t,x')\big),
\quad\lim\limits_{\yinf}\ta(t,x,y)=\Ta(t,x'),\\
(\bu_h,\ta)|_{t=0}=(\bu_{h0},\ta_0)(x',y).
\end{cases}\end{equation}

\begin{remark} In \cite{LWY4}, the authors have studied the small viscosity and heat conductivity limit for the compressible Navier-Stokes-Fourier equations with nonslip boundary condition on velocity and the same condition as given in \eqref{bd_oreq} for the temperature, and obtained that the thermal layer profiles satisfy the same problem as given in \eqref{pr_bd1},  when the viscosity goes to zero faster than the heat conductivity.

\end{remark}


{\bf Acknowledgements:}
The first two authors' research was supported in part by
National Natural Science Foundation of China (NNSFC) under Grant Nos. 91230102 and 91530114, and the second author's research was also supported by Shanghai Committee of Science and Technology under Grant No. 15XD1502300. The last author's research was supported by the General Research Fund of Hong Kong,
CityU No. 103713.

\end{document}